\documentclass[12pt]{article}

\usepackage{amsmath}
\usepackage{amsfonts}
\usepackage{latexsym}
\usepackage{amssymb}
\usepackage{enumerate}
\usepackage{amsthm}
\usepackage[normalem]{ulem}
\usepackage{csquotes}
\usepackage{lineno,xcolor}
\usepackage{scalefnt}
\usepackage{tikz}
\usepackage{float}
\usepackage{fancyhdr}
\usepackage{mathabx}
\usetikzlibrary{arrows}
\usetikzlibrary{positioning}
\usepackage{caption}

\definecolor{amaranth}{rgb}{0.9, 0.17, 0.31}
\definecolor{bluegray}{rgb}{0.4, 0.6, 0.8}

\makeatletter

\newtheorem*{maintheorem*}{Main Theorem}
\newtheorem{theorem}{Theorem}[section]
\newtheorem{proposition}[theorem]{Proposition}
\newtheorem{corollary}[theorem]{Corollary}
\newtheorem{lemma}[theorem]{Lemma}

\newtheorem*{theorem*}{Theorem}

\newtheorem{remark}[theorem]{Remark}
\newtheorem*{remark*}{Remark}

\newtheorem*{example*}{Example}
\newtheorem*{conjecture*}{Conjecture}
\newtheorem*{problem*}{Problem}
\def\1{\mathbf 1}

\def\0{\mathbf 0}
\def\cA{\mathcal A}

\def\cD{\mathcal D}

\def\cM{\mathcal M}

\def\cX{\mathcal X}

\def\cT{\mathcal T}

\def\cR{\mathcal R}
\def\cS{\mathcal S}

\def\bC{{\mathbb C}}
\def\bF{{\mathbb F}}

\def\bN{{\mathbb N}}

\def\bQ{{\mathbb Q}}
\def\bR{{\mathbb R}}
\def\bZ{{\mathbb Z}}

\def\Fp{\mathbb F_{p}}

\def\deg{{\rm deg}}
\def\det{{\rm det}}

\def\dim{{\rm dim }}

\def\End{{\rm End}}

\def\GCD{{\rm GCD}}

\def\Mat{{\rm Mat}}
\def\mod{{\rm mod}}

\def\Span{\text{\rm Span}}
\def\sgn{{\rm sgn}}

\def\Trace{{\rm Trace}}

\newcommand{\floor}[1]{\left\lfloor #1 \right\rfloor}
\def\BOUND{B_r}
\def\<{\langle}
\def\>{\rangle}
\newcommand\comment[1]{}

\newcommand*{\shifttext}[2]{
  \settowidth{\@tempdima}{#2}
  \makebox[\@tempdima]{\hspace*{#1}#2}
}
\newlength{\tagwidth}
\newcommand\redsout{\bgroup\markoverwith{\textcolor{amaranth}{\rule[0.5ex]{2pt}{0.4pt}}}\ULon}

\newcommand\redout{\bgroup\markoverwith
{\textcolor{red}{\rule[.4ex]{2pt}{0.8pt}}}\ULon}

\makeatletter
\def\@fnsymbol#1{\ensuremath{\ifcase#1\or \ast\or\dagger\or \ddagger\or
		\mathsection\or \mathparagraph\or \|\or **\or \dagger\dagger
		\or \ddagger\ddagger \else\@ctrerr\fi}}
\makeatother

\title{Unbalanced distance-biregular graphs \\ with girth deficiency two}
\author{Sebastiano Argenti 
	\footnote{Sebastiano Argenti: sargenti@mun.ca\hfill\newline Department of Mathematics and Statistics, Memorial University of Newfoundland, St. John’s, NL, A1C5S7, Canada
	}
\and
Alessandro Siciliano\footnote{Alessandro Siciliano: alessandro.siciliano@unibas.it\hfill\newline Dipartimento di Scienze di Base e Applicate -	Universit\`{a} degli Studi della Basilicata - Viale dell'Ateneo Lucano 10 - 85100 Potenza (Italy).}
}
\date{}

\begin{document}

\maketitle

\thispagestyle{fancy}

\fancyhf{}
\renewcommand{\headrulewidth}{0pt}
\lhead{}

\begin{abstract}
Let $\Gamma=(V,E)$ be a distance-biregular graph with even diameter $d$ and girth $2d-2$. On the edge-set $E$ we define relations that form an association scheme; this is done by considering the linear span of the corresponding adjacency matrices, and, by means of intersection diagrams, we prove that this is an algebra.  Furthermore, we describe all the irreducible representations of this algebra, and, by considering the multiplicities of the 2-dimensional ones, we show that there are no such distance-biregular graphs with $28\le d\le 46$ and $d\ge 52$.
\end{abstract}

{\it Keywords: Distance-biregular graph, Association scheme, Irreducible representation}             

{\it Math. Subj. Class.: 05C50,05E30,05E10} 

\section{Introduction} \label{sec_1}

Throughout this paper, $\Gamma=(V,E)$ is a finite undirected graph with vertex set $V$ and edge set $E$, the latter consisting of unordered pairs of distinct vertices. 




For any vertex $u$ of a connected graph $\Gamma$,  let $\varepsilon(u)$ be the  eccentricity of $u$, i.e., the maximum distance from $u$ to any other vertex. For  $i=0,1,2,\ldots,\varepsilon(u)$, let $\Gamma_i(u)$  be the set of vertices at distance $i$ from $u$.  The vertex $u$ is said to be {\em distance-regularized} if, for each vertex $v$, the numbers
\[
a_i(u)=|\Gamma_i(u)\cap\Gamma_1(v)|, \ \ \ b_i(u)=|\Gamma_{i+1}(u)\cap\Gamma_1(v)|, \ \ \ c_i(u)=|\Gamma_{i-1}(u)\cap\Gamma_1(v)|,
\]
for $i=0,\ldots,\varepsilon(u)$, depend only on the distance $d(u,v)=i$ and are independent of the choice of $v$ in $\Gamma_i(u)$.  The numbers $a_i(u)$, $b_i(u)$ and $c_i(u)$ are called the $i${\em-intersection numbers at} $u$, and  the array
 \[
 \begin{pmatrix}
 * & c_1(u) & \cdots & c_{\varepsilon-1}(u) & c_\varepsilon(u) \\
 0 & a_1(u) & \cdots & a_{\varepsilon-1}(u) & a_\varepsilon(u) \\
 b_0(u) & b_1(u) & \cdots & b_{\varepsilon-1}(u) & * \\
 \end{pmatrix}
 \]
 is the {\em intersection array at} $u$. 

A {\em distance-regularized graph} is a connected graph in which every vertex is distance-regularized. There are two very important families of distance-regularized graphs: distance-regular graphs, in which each vertex has the same intersection array, and  distance-biregular graphs, which are bipartite graphs with the property that vertices in the same partition class have the same intersection array.  

The concept of a distance-regularized graph was introduced by Godsil and Shawe-Taylor \cite{gst}  to create a common framework that describes both distance-regular graphs and incidence graphs of generalized polygons, the latter determining, in a natural way, two distance-regular graphs. In the same paper, the authors proved the following result (see \cite{fiol1} for an alternative proof).
\begin{theorem*}\cite{gst}
Let $\Gamma$ be a distance-regularized graph. Then $\Gamma$ is either  distance-regular or distance-biregular.
\end{theorem*}

Unlike the theory of distance-regular graphs, which remains one of the most active fields of research in algebraic combinatorics, the theory of distance-biregular graphs appears to have made limited progress since its emergence about forty years ago.  This is despite the distinctive character of these graphs due to their rich connections with many combinatorial structures. 

As meaningfully presented in the monograph \cite{bcn}, distance-regular graphs have strong connections to many parts of graph theory, coding theory, design theory and geometry. The first to recognize and fully use these graphs as the basic structures underlying most finite objects of sufficient regularity was Delsarte \cite{dels}. Besides their connections with other areas, distance-regular graphs are also important mathematical objects in their own right.

Distance-biregular graphs are closely related to distance-regular graphs. In fact, the square of a connected distance-biregular graph induces a distance-regular graph on each of its bipartition classes. Thus, roughly speaking, distance-biregular graphs can be viewed as ``{\em a pair of distance-regular graphs that fit well together}''.

There are numerous constructions of distance-biregular graphs. 
A first family is obtained as follows. Consider the set 
$X=\{1,\ldots,n\}$ and let $k<n$. Let $A$ and $B$ be the sets of 
$k$-subsets and $(k+1)$-subsets of $X$, respectively. The set 
$V=A\cup B$, with adjacency defined by inclusion, gives a 
distance-biregular graph whose diameter is $2k+1$ if $n=2k+1$, and 
$2k+2$ if $n\geq 2k+2$. Since the halved graphs are the Johnson graphs 
$J(n,k)$ and $J(n,k+1)$ \cite{bcn}, these graphs are called 
{\em BiJohnson graphs} and are denoted by $BJ(n,k)$. Their $q$-analogs 
are the {\em BiGrassmann graphs}, denoted by $BJ_q(n,k)$ 
\cite{del,gst,lato2}. The graphs in these families have arbitrarily 
large diameter and girth six.

Another important source of distance-biregular graphs comes from 
incidence geometry. A distance-biregular graph can be viewed as a 
point-block incidence structure \cite{del}, where the edges of the graph 
correspond to the flags of the structure\footnote{A {\em flag} of an 
incidence structure is an incident point-block pair.}. In particular, 
distance-biregular graphs with girth twice the diameter are well 
understood. Examples are provided by the incidence graphs of generalized 
polygons, which are point-line incidence structures introduced by Tits 
\cite{tit} in the study of groups of Lie type. Conversely, any bipartite 
graph with girth twice the diameter and minimum degree at least three is 
the incidence graph of a generalized polygon \cite[Lemma 1.3.6]{vm}. 
Along these lines, Delorme \cite{del} and Godsil and Shawe-Taylor 
\cite{gst} gave many examples of distance-biregular graphs arising from 
combinatorial designs and partial geometries.

There are also spectral characterizations of these graphs. In 
\cite[Theorem 5.2]{lato3}, Lato showed that a semiregular bipartite graph 
with diameter $d$, $d+1$ distinct eigenvalues, and girth $2d-2$ is 
distance-biregular. This result is closely related to the problem raised 
in \cite[Problem 5.5.2]{lato2} of determining the possible values of the 
diameter $d$ of a distance-biregular graph with girth $2d-2$. The case 
$d=4$, and hence girth $6$, is particularly significant: such a graph is 
the incidence graph either of a partial geometry of type 
$PG(s,t,\alpha)$, with $\alpha\geq2$, or of a Steiner system of type 
$S(2,k,v)$ \cite{del,st1}, according to whether or not all vertices have 
the same eccentricity. For $d=6$ and girth $10$, Delorme \cite{del} 
constructed an example given by the subdivision graph of the Petersen 
graph.

Beyond these geometric and spectral connections, distance-biregular 
graphs are also related to Hadamard matrices and $2$-fold covers of 
complete graphs \cite[Theorem 3.1]{st2}. For further results on 
distance-biregular graphs, see 
\cite{acej,del,fk,fp,fiol1,fiol2,hp,lato3,lato1,nom}; for a more 
extensive treatment, we refer the reader to \cite{lato2}.  
In her PhD thesis \cite{lato2} (see also \cite{lato4}), Lato showed that the vertex set of any distance-biregular graph can be endowed with the structure of a coherent configuration with two fibers \cite{cp}. When the distance-biregular graph is a generalized polygon, Higman \cite{hig}   showed that it is possible to define an association scheme on the edges of the graph. By studying the representations of the associated  adjacency algebra, Higman provided another proof of the result that the incidence graph of a generalized polygon with minimum degree three has diameter $d=2,3,4,6$ or $8$. The original proof of this result was given by Feit and Higman in \cite{fh} by using a different approach.  

Taking inspiration from the work of Higman \cite{hig}, in this paper we consider distance-biregular graphs with girth $2d-2$, where $d$ is the diameter of the graph. We show that, when the eccentricities of the vertices take two distinct values, it is possible to define certain relations on the edges of the graph that give an association scheme. Then, by studying the representations of the adjacency algebra associated with this scheme, and the factorization of certain polynomials over the integers, we prove the following result:

\begin{maintheorem*}
Let $\Gamma=(V,E)$ be a distance-biregular graph with diameter $d$, girth $2d-2$, and with two distinct eccentricities. Then, $d$ is even with $4\le d \le 26$ or $d\in\{48,50\}$.
\end{maintheorem*}
 
 The structure of the paper is as follows. Section \ref{sec_2} contains some notation and introductory material on  distance-biregular graphs, association schemes and the representations of the associated adjacency algebra.  We  introduce the notion of unbalanced distance-biregular graph with girth deficiency two and  we investigate some properties of the parameters of such graphs. 
In Section \ref{sec_3}, for any such  graph,  we define relations on its edge-set and show that they form an association scheme. This is done by considering the linear span of the corresponding adjacency matrices and, by means of intersection diagrams, we prove this is actually an algebra of matrices.
The irreducible representations of this algebra are completely described in Section \ref{sec_4}. It turns out that they can be 1- or 2-dimensional, and those of dimension two are in one-to-one correspondence with the zeros of a polynomial over the integers, whose degree is related to the diameter of the graph. 
The multiplicities of irreducible representations are calculated in Section \ref{sec_5}. {We have reduced this calculation to a sum of rational functions, which is then performed on a computer.} The expression of the multiplicity of the two-dimensional irreducible representation provides information on the factorization of the above polynomial over the integers. It turns out that the degree of its irreducible factors over the integers is at most two. In Section \ref{sec_6}, by taking this information into account, we are able to provide the spectrum for the values of the diameter of an unbalanced distance-biregular graph with girth deficiency two. Once again, we reduced this problem to a finite calculation, which was then performed on a computer. Finally, Section \ref{sec_7} contains a more detailed analysis of the restrictions on the parameters of the graph. More precisely, for each feasible value of the diameter, we studied the possible values of the valencies of the vertices of such graphs.  The results obtained lead us to conjecture that there are no
unbalanced distance-biregular graphs with diameter greater than six
and girth deficiency two.

Our approach is similar to that of \cite{hig}, in that we endow the
edge set of an unbalanced distance-biregular graph with girth
deficiency two with the structure of an association scheme. However,
our most significant results rely on properties of certain orthogonal
polynomials and on methods from modular arithmetic similar to those
introduced by Fuglister \cite{fug}.

\section{Preliminaries and first results} \label{sec_2}

Let $\Gamma=(V,E)$ be a finite undirected graph with vertex set $V$ and edge set $E$. To make shorter the notation, the edge with vertices $u$ and $v$ will be denoted by $uv$. Two vertices $u$ and $v$ are said to be {\em adjacent} if $uv\in E$.

A {\em non-stammering walk} of length $n$ in $\Gamma$ is a sequence  $u_0,u_1,\ldots,u_{n-1},u_n $ of vertices, where, for $1\le i \le n$, $u_{i-1}$ and $u_i$ are adjacent  and  $u_{j-1}\neq u_{j+1}$, for $j=1,\ldots, n-1$.  We will use simply the word walk when no confusion can arise. A {\em circuit} is a  walk where $u_0=u_n$. A walk containing no repeated vertices is a {\em path}. A circuit
$u_0,u_1,\ldots,u_n$ is a {\em cycle} if
$u_0,u_1,\ldots,u_{n-1}$ are pairwise distinct. If $\Gamma$ has a
cycle, then the {\em girth} $g(\Gamma)$ is the minimum length of a
cycle in $\Gamma$; if $\Gamma$ has no cycle, then its girth is infinite.

A graph $\Gamma$ is {\em connected} if there is a path between every pair of its vertices.   In a connected graph $\Gamma$, the {\em distance} between two vertices $u$ and $v$ is defined to be the minimum length  $d(u,v)$ of a path from $u$ to $v$. For any vertex $u$ of a connected graph $\Gamma$,  the  {\em eccentricity $\varepsilon(u)$ of}  $u$ is the maximum distance from $u$ to any other vertex. The {\em diameter} $d(\Gamma)$ of $\Gamma$ is the maximum eccentricity
among its vertices. If $\Gamma$ contains a cycle, then
$g(\Gamma)\le2d(\Gamma)+1$.

A graph $\Gamma$ is {\em bipartite} if its vertex set $V$ can be partitioned into two non-empty disjoint subsets  $A$ and $B$ such that any edge of $\Gamma$ joins a vertex in $A$ with a vertex in $B$; the sets $A$ and $B$ are called the {\em partition classes} of $\Gamma$.

 For  $i=0,1,2,\ldots,\varepsilon(u)$, let $\Gamma_i(u)$  be the set of vertices at distance $i$ from $u$. The size of $\Gamma_1(u)$ is  the {\em degree of} $u$:  $\deg(u)=|\Gamma_1(u)|$. A vertex $u$ is said to be {\em distance-regularized} if, for each vertex $v$, the numbers
\[
a_i(u)=|\Gamma_i(u)\cap\Gamma_1(v)|,\qquad
b_i(u)=|\Gamma_{i+1}(u)\cap\Gamma_1(v)|,\qquad
c_i(u)=|\Gamma_{i-1}(u)\cap\Gamma_1(v)|,
\]
for $i=0,\ldots,\varepsilon(u)$, depend only on the distance $d(u,v)=i$ and are independent of the choice of $v$ in $\Gamma_i(u)$.  The numbers $a_i(u)$, $b_i(u)$ and $c_i(u)$ are called the $i${\em-intersection numbers at} $u$, and  the array
 \[
 \begin{pmatrix}
 * & c_1(u) & \cdots & c_{\varepsilon-1}(u) & c_\varepsilon(u) \\
 0 & a_1(u) & \cdots & a_{\varepsilon-1}(u) & a_\varepsilon(u) \\
 b_0(u) & b_1(u) & \cdots & b_{\varepsilon-1}(u) & * \\
 \end{pmatrix}
 \]
 is the {\em intersection array at} $u$.  Clearly, $b_0(u)=\deg(u)$. 
 A {\em distance-regularized graph} is a connected graph in which every vertex is distance-regularized. According to the result of Godsil and Shawe-Taylor \cite[Theorem 2.2, Corollary 2.3.1]{gst}, the only distance-regularized graphs are distance-regular graphs or distance-biregular graphs.

Let $\Omega$ be a finite set of size $n$, and $\cR$ a partition of $\Omega\times \Omega$. The elements of $\cR$ are called {\em relations on} $\Omega$, and for a relation $s$, the relation $s'=\{(y,x):(x,y)\in s\}$ is the {\em transpose} of $s$. 

An {\em association scheme} on $\Omega$  is a pair $\cX=(\Omega,\cR)$ such that 
\begin{enumerate}
\item[(AS1)] $1_\Omega=\{(x,x):x \in \Omega\}\in \cR$;
\item[(AS2)] $s' \in \cR$, for all $s\in \cR$;
\item[(AS3)] for any $r,s,t\in \cR$, the number $p_{s\,t}^r$ of $z\in \Omega$ such that $(x,z)\in s$ and $(z,y)\in t$ does not depend on the choice of $(x,y)\in r$. 
\end{enumerate}
 
The cardinality of $\Omega$ is the {\em order} of $\cX$, and the integers $p_{s\,t}^r$ are the {\em intersection numbers} of $\cX$; for every relation $s$, $\eta_s=p^{1_\Omega}_{s,s'}$ is the {\em valency} of $s$. Note that $|\Omega|=\sum_{s\in \cR}\eta_s$.

In the following, as usual, $\bC$ is the field of complex numbers, 
$\bC^\Omega$ is the vector space of all $|\Omega|$-tuples over $\bC$, 
whose entries are indexed by the elements of $\Omega$, and 
$\Mat_\Omega(\bC)$ is the matrix algebra over $\bC$ whose rows and 
columns are indexed by the elements of $\Omega$. We set $I$ and $J$ 
to be the identity matrix and the all-ones matrix in 
$\Mat_\Omega(\bC)$, respectively. For any $A\in\Mat_\Omega(\bC)$, 
we denote its transpose by $A^\top$ and its Hermitian transpose by 
$A^*$. With $\<\ ,\ \>$ we denote the canonical Hermitian inner 
product on $\bC^\Omega$.

For any subset $\cM$ of $\Mat_\Omega(\bC)$,  $\Span_\bC(\cM)$ is the linear span of $\cM$ in $\Mat_\Omega(\bC)$.
The {\em Hadamard} (or {\em Schur}) {\em product} $A\circ B$ of the matrices $A,B\in \Mat_\Omega(\bC)$ is defined by the formula
\[
(A\circ B)_{xy}=A_{xy}B_{xy}.
\]

The {\em adjacency matrix} of a relation $s\subseteq \Omega\times\Omega$	 is defined to be the $\{0,1\}$-matrix $A_s\in \Mat_\Omega(\bC)$ such that $(A_s)_{xy}=1$ if and only if $(x,y)\in s$.

Let $\cX=(\Omega,\cR)$ be an association scheme. The subset $\cM(\cX)=\{A_s:s\in \cR\}$ of $\Mat_\Omega(\bC)$   consists of pairwise orthogonal matrices with respect to the Hadamard
product. Then, $\cM(\cX)$ is a basis of  $\cA=\Span_\bC(\cM(\cX))$. From property (AS3) it follows that $\cA$ is a subalgebra of $\Mat_\Omega(\bC)$, whose dimension is the size of the set $\cR$; the subalgebra $\cA$ is  the {\em adjacency  algebra} (or the {\em Bose-Mesner  algebra}) of the association scheme  $\cX$.
\begin{theorem}\cite[Corollary 2.3.8]{cp}\label{th_2}
Let $\cM=\{A_0=I,A_1,\ldots,A_d\}$ be a finite family of non-zero $n\times n$ $\{0,1\}$-matrices. Then, $\Span_{\bC}(\cM)$ is the adjacency algebra of an association scheme (on a set $\Omega$ of size $n$) if and only if the following conditions hold:
\begin{itemize}
\item[i)]  $\sum_{i=0}^{d}{A_i}=J$;
\item[ii)] for all $A_i\in \cM$, $A_i^\top =A_{i'}$, for some $i'\in\{0,\ldots, d\}$;
\item[iii)] $A_iA_j=\sum_{k=0}^{d}{p^k_{ij}A_k}$, for all $i,j\in\{0,\ldots, d\}$.
\end{itemize} 
\end{theorem}

Note that property iii) can be interpreted as meaning that $\cM$ spans a subalgebra of $\Mat_\Omega(\bC)$.
The reader is referred to \cite{bi,cp,dels} for additional information on association schemes.

Let $\cA$ be the adjacency algebra of an association scheme $\cX=(\Omega,\cR)$. A {\em representation} of $\cA$ is a pair $(U,\rho)$, where $U$ is a vector space  and $\rho$ is a homomorphism $\rho:\cA\rightarrow \End\,(U)$. The representation  $(U,\rho)$ is {\em irreducible} if $U$ is an irreducible module, that is, it does not contain proper non-trivial submodules.

The algebra $\cA$ acts by multiplication (on the left) on $\bC^\Omega$, which is called the {\em standard module} of $\cA$.
Let $U\subset\bC^\Omega$ be an $\cA$-submodule, i.e. $AU\subseteq U$, for all $A\in \cA$. This yields that every $A\in\cA$ acts as an endomorphism of $U$. 
 For any $A\in\cA$, $v\in U$ and $w\in U^\perp$ we have $\<Aw,v\>=\<w,A^*v\>=0$.
This yields that the subspace  $U^\perp$ is also $\cA$-invariant. 
It follows that $\bC^\Omega$ can be written as a direct sum of irreducible mutually orthogonal submodules. Let $U$ be one such  submodule. Since $U$ is irreducible,  the homomorphism $\rho:\cA\rightarrow \End\,(U)$ is surjective by the Density Theorem \cite[Theorem 3.2.2]{eghls}. Once a basis for $U$ is fixed,  $\End\,(U)$ can be  canonically identified with $\Mat_{\dim\, U}(\bC)$. Therefore, $\rho(\cA)$ is actually the full matrix algebra $\Mat_{\dim\, U}(\bC)$.  In addition,  if the chosen basis is orthonormal we have $\rho(A^*)=\rho(A)^*$, for any $A\in\cA$.
  
Two  representations $(U,\rho)$ and $(U',\rho')$ of $\cA$ are {\em isomorphic} if there is a bijective intertwining operator $\varphi: U \rightarrow  U'$,  that  is a linear isomorphism which commutes with the action of $\cA$: $\varphi(\rho(A)v) = \rho'(A)(\varphi(v))$, for any $v \in U$ and $A\in\cA$. 
 
For any $x\in \Omega$, let  $\chi_x\in\bC^\Omega$  be the vector with 1 in the entry corresponding to $x$ and 0 in the other entries. The set $\{\chi_x:x\in \Omega\}$ is the {\em canonical basis} of $\bC^\Omega$. 
For $x\in \Omega$ and $s\in \cR$, we have $A_s\chi_x=\sum_{\substack{y\in \Omega\\(y,x)\in s}}{\chi_y}$. Therefore, $A_sw=\eta_sw$, where $w=\sum_{x\in \Omega}{\chi_x}$ is the all-ones vector and $\eta_s$ is the valency of $s$. 
Thus, $U=\<w\>$ is an $\cA$-submodule, and it is clearly irreducible;  it is  called the {\em principal representation} of $\cA$, whose associated homomorphism is $\rho(A_s)=\eta_s$ for $s\in\cR$.

 {Let $(U_i,\rho_i)$, $i=1,\ldots,t$, } be the nonisomorphic irreducible representations of $\cA$. Since $\cA$ is semisimple \cite{hig3}, the standard module $\bC^\Omega$ decomposes as 
\[
\bC^\Omega=U_1\perp\underbrace{U_2\perp\ldots\perp U_2}_{m_2}\perp\cdots\perp \underbrace{U_t\perp\ldots\perp U_t}_{m_t},
\]
for some positive integers $m_i$, called  the multiplicity of $U_i$. In addition, 
\begin{equation}\label{eq_28}
(\dim\, U_1)^2+\ldots+(\dim\, U_t)^2=\dim\,\cA
\end{equation} 
\cite[Proposition 3.5.8]{eghls}. We refer the reader to \cite{eghls} for more on representation theory.

Let $\Gamma=(V,E)$ be a distance-biregular graph with partition classes $A$ and $B$.
We denote the eccentricity of the vertices in $A$ (resp. in $B$) by $d_A$ (resp. $d_B$). Clearly, $d(\Gamma)=\max\{d_A,d_B\}$, and, without loss of generality, we may assume $d(\Gamma)=d_A$. To shorten notation, we will write $d$ instead of $d(\Gamma)$ and $g$ instead of $g(\Gamma)$.

For $u\in A$ and $v\in\Gamma_{i}(u)$, let
\[
\beta_i^{A}=|\Gamma_{i+1}(u)\cap\Gamma_1(v)|,
\qquad 0\le i<d_A,
\]
\[
\gamma_i^{A}=|\Gamma_{i-1}(u)\cap\Gamma_1(v)|,
\qquad 0<i\le d_A.
\]

The integers $\beta_i^{B}$, for $0\le i< d_B$, and $\gamma_i^{B}$, for $0< i\le d_B$, are defined in the same way for a vertex $u\in B$.  We have $\gamma_1^{A}=\gamma_1^{B}=1$ and 
\[ 
\deg(u)=\left\{
\begin{array}{ll}
\beta_0^A, &  \mathrm {if\ } u\in A\\[.03in]
\beta_0^B, &\mathrm {if\ } u\in B
\end{array}
\right..
\]

Furthermore,  for $\{\pi,\xi\}=\{A,B\}$, we have
\begin{equation}\label{eq_1}
\beta_i^\pi+\gamma_i^\pi=
\left\{\begin{array}{cl}
\beta_0^\pi & \mbox{if $i$ is even}\\[.03in] 
\beta_0^\xi & \mbox{if $i$ is odd}
\end{array}
\right.,
\end{equation}
and 
\begin{equation}\label{eq_2}
\gamma_{d_\pi}^\pi=
\left\{\begin{array}{cl}
\beta_0^\pi & \mbox{if $d_\pi$ is even}\\[.03in] 
\beta_0^\xi & \mbox{if $d_\pi$ is odd}
\end{array}
\right..
\end{equation}

The integers $\beta_i^A$, $\gamma_i^A$, $\beta_i^A$ and $\gamma_i^B$ are the {\em intersection numbers} of  $\Gamma$. They are usually arranged in a two-line array:

The integers $\beta_i^A$, $\gamma_i^A$, $\beta_i^B$ and $\gamma_i^B$
are the {\em intersection numbers} of $\Gamma$. They are usually
arranged in a two-line array:
\[
\begin{array}{|c|}
\beta_0^A;\ \  1,\ \gamma_2^A,\ \cdots, \ \gamma_{d_A}^A\\[.05in]
\beta_0^B;\ \ 1,\ \gamma_2^B,\ \cdots, \ \gamma_{d_B}^B
\end{array}\ .
\]
This is the {\em intersection array} of $\Gamma$.

 The following characterization was  given by Mohar and Shawe-Taylor  \cite[Corollary 3.5]{mst}; see also  \cite[Theorem  3.6.2]{lato2}.

\begin{theorem}\label{th_6}
Let $\Gamma$ be a distance-biregular graph with vertices of degree two. Then, {either} $\Gamma$ is the complete bipartite graph $K_{2,n}$ or {$\Gamma$ is the subdivision graph of  a Moore graph or the subdivision graph of the incidence graph of a regular generalized polygon.}
\end{theorem}
\comment{
From now on we make the following assumpion:
\begin{equation}\label{eq_30}\tag*{(\textasteriskcentered{})} 
\begin{array}{l}
\Gamma=(V,E) \mathrm{\ is\ a\  unbalanced\ distance-biregular\ graph\  with\ girth\ deficiency\ two.}
\end{array}
\end{equation}
}
A distance-biregular graph with minimum degree at least three is said to be {\em thick}. 

{Let $\Gamma$ be a biregular graph with diameter $d$ and girth $g$ that is not  the complete bipartite graph  $K_{1,n}$. Then $g\le 2d$.} 
We will call the difference  $2d-g$ the {\em girth deficiency} of $\Gamma$. 

By \cite[Lemma 1.3.6]{vm}, distance-biregular graphs with girth
deficiency zero are the incidence graphs of generalized polygons, and
the thick ones have diameter $d=2,3,4,6$ or $8$ \cite{fh,hig}.
Yanushka \cite{yan} proved that a non-thick distance-biregular graph
with girth deficiency zero is the $k$-fold subdivision of a multiple
edge or the $k$-fold subdivision of a thick generalized polygon.
The case of girth deficiency two has been studied for bipartite distance-regular graphs only: they have diameter $d\leq 14$, $d\neq 11$ \cite{ckn}. 
For this reason we are interested {in} distance-biregular graphs with girth deficiency two. 

The following result is rather straightforward.

\begin{lemma}\cite[Proposition 1]{del}\label{lem_1}
Let $d=d_A\ge d_B$. If  $d_A$ is even  then either $d_B=d_A$ or $d_B=d_A-1$; if $d_A$ is odd  then $d_B=d_A$.
\end{lemma}

A distance-biregular graph $\Gamma$ with $d_B=d_A$ is said to be {\em balanced}; it is said to be {\em unbalanced} otherwise. 

\begin{proposition}\label{lem_2}
Let $\Gamma$ be an unbalanced distance-biregular graph with girth deficiency two and partition classes $A$ and $B$. Then, the intersection array of $\Gamma$ is 
\[
\begin{array}{|cccc|}
\beta_0^A; &   1,\cdots,   1, &\beta_0^A, & \beta_0^A \\[.05in]
\beta_0^B; &  \underbrace{1, \cdots,  1,}_{d-2} & \beta_0^A & 
\end{array}.
\]
\end{proposition}
\begin{proof}
Assume $\gamma_i^{\pi}\ge 2$, with $\pi\in\{A,B\}$. By using the definition of $\gamma_i^{\pi}$, it is easily seen that in $\Gamma$ there is a circuit of length $2i$. By the assumption on the girth we have $2i\ge g=2d-2$, that is $i\ge d-1$. Therefore, $\gamma_i^{\pi}=1$, for $i=1,\ldots,d-2$. 
 
From the equality \eqref{eq_2} we get $\gamma_{d}^A=\beta_0^A=\gamma_{d-1}^B$, and, from \cite[Proposition 3]{del}, we have
\[
\gamma_{d-2}^A\gamma_{d-1}^A=\gamma_{d-2}^B\gamma_{d-1}^B,
\]
which gives $\gamma_{d-1}^A=\beta_0^A$.
\end{proof}
Since we assume  $d=d_A$, from Lemma \ref{lem_1} it is evident that the diameter $d$ of an unbalanced distance-biregular graph $\Gamma$ with girth deficiency two is even, and Proposition \ref{lem_2} provides its intersection array. In the following, to make the notation easier, we set 
\[
d=2r+2\ (r\ge 1), \ \ \ \ \ \ \
k=\beta^A_0-1,\ \ \ \ \ \ \ l=\beta^B_0-1.
\]
{The following result is due to Lato.
\begin{lemma}\label{lem_19}\cite[Lemma 2.4.4]{lato2}
Let $\Gamma=(V,E)$ be a distance-biregular graph with partition classes $A$ and $B$ and intersection array
\[
\begin{array}{|c|}
\beta_0^A;\ \  1,\ \gamma_2^A,\ \cdots, \ \gamma_{d_A}^A\\[.05in]
\beta_0^B;\ \ 1,\ \gamma_2^B,\ \cdots, \ \gamma_{d_B}^B
\end{array}\ .
\]
Then, we have  $|\Gamma_i(u)|\gamma_i^\pi=|\Gamma_{i-1}(u)|\beta_{i-1}^\pi$  for each $\pi\in\{A,B\}$, $u\in\pi$ and $i=1,\ldots,\varepsilon(u)$.
\end{lemma}
}
We now apply Lemma \ref{lem_19} to the intersection array given in
Proposition \ref{lem_2}. This yields explicit formulas for the remaining
intersection numbers and for the sizes of the sets $\Gamma_i(u)$, the
bipartition classes, and the edge set.
\begin{proposition}\label{cor_1}
Let $\Gamma=(V,E)$ be an unbalanced distance-biregular graph with girth deficiency two and partition classes $A$ and $B$. Then,
\begin{enumerate}
\item
\[
\begin{array}{c}
\begin{array}{lcl}
\beta_{2j-1}^A & = & l, \\[.06in]
\beta_{2j}^A & = & k,
\end{array} \ \begin{array}{lcl}
\beta_{2j-1}^B & = & k, \\[.06in]
\beta_{2j}^B & = & l,
\end{array} \ \ \ \mbox{ for $j=1,\ldots,r$}; \\[.35in]
\beta_{d-1}^A  =  l-k.
\end{array}
\]
\item
If $u\in A$, then
\[
\begin{array}{c}
\begin{array}{lcl}
|\Gamma_{2j-1}(u)|=k^{j-1}l^{j-1}(k+1)  \\[.1in]
|\Gamma_{2j}(u)|=k^{j-1}l^{j}(k+1)
\end{array} \ \ \ 
\mbox{ for $j=1,\ldots,r$};\\[.35in]
|\Gamma_{d-1}(u)|=k^rl^{r};\\[.1in]
|\Gamma_{d}(u)|=\frac{k^rl^{r}(l-k)}{k+1}.
\end{array}
\]
If $u\in B$, then
\[
\begin{array}{l}
\begin{array}{lcl}
|\Gamma_{2j-1}(u)|=k^{j-1}l^{j-1}(l+1),  \\[.1in]
|\Gamma_{2j}(u)|=k^{j}l^{j-1}(l+1), 
\end{array} \ \ \ 
\mbox{ for $j=1,\ldots,r$};\\[.35in]
|\Gamma_{d-1}(u)|=\frac{k^rl^{r}(l+1)}{k+1} .
\end{array}
\]
\item
\[
|A|=\frac{(k^{r+1}l^r(l+1)-k-1)(l+1)}{(k+1)(kl-1)},
\]
\[
|B|=\frac{k^{r+1}l^r(l+1)-k-1}{kl-1},
\]
\[
|E|=\frac{(k^{r+1}l^r(l+1)-k-1)(l+1)}{kl-1}.
\]
\end{enumerate}
\end{proposition}
\begin{proof}
The values of $\beta^A_i$ and $\beta^B_i$  are obtained from Proposition \ref{lem_2} by using \eqref{eq_1}. 

Let $u\in\pi$, with $\pi\in\{A,B\}$. From {Lemma \ref{lem_19}}, we have  $|\Gamma_i(u)|\gamma_i^\pi=|\Gamma_{i-1}(u)|\beta_{i-1}^\pi$, for  $i=1,\ldots,\varepsilon(u)$. By taking into account the  values of $\beta^\pi_i$ and those of $\gamma^\pi_i$ given by Proposition \ref{lem_2} and applying recursively the above formula, we find the values  of $|\Gamma_{i}(u)|$, for $i=1,\ldots,\varepsilon(u)$.

 For a fixed vertex $u\in A$, the partition class $A$ is the disjoint
union of the sets $\Gamma_{2j}(u)$, for $j=0,\ldots,r+1$, whereas, for
a fixed vertex $u\in B$, the partition class $B$ is the disjoint union
of the sets $\Gamma_{2j}(u)$, for $j=0,\ldots,r$. Thus,
\[
|A|=\sum_{j=0}^{r+1}|\Gamma_{2j}(u)| \quad (u\in A),
\qquad
|B|=\sum_{j=0}^{r}|\Gamma_{2j}(u)| \quad (u\in B).
\]

  Finally, $|E|=|A|\beta_0^A=|B|\beta_0^B$.
\end{proof}
\begin{corollary}\label{cor_2}
Let $\Gamma$ be an unbalanced distance-biregular graph with girth deficiency two. Then, the parameters $r$, $k$ and $l$ satisfy the following feasibility conditions:
\begin{enumerate}
\item\label{item1} $k<l$;
\item $l^r(l+1)\equiv 0\,\mod\, (k+1)$.
\end{enumerate}
\end{corollary}
\begin{proof}
Since $\beta_{d-1}^A$ must be a positive integer,  $k<l$. 
From the size of $\Gamma_{d}(u)$, with $u\in A$, or equivalently from $\Gamma_{d-1}(u)$, with $u\in B$, we get $l^r(l+1)\equiv 0\,\mod\, (k+1)$, as $k\,\equiv -1\, \mod\, (k+1)$.
\end{proof}
\begin{corollary}\label{cor_3}
Let $d\ge 2$ and let $q$ be a positive rational number. Then there are only finitely many unbalanced distance-biregular graphs
with diameter $d$, girth deficiency two and degrees $k+1$ and $l+1$
such that $k/l=q$.
\end{corollary}
\begin{proof}
Note that, by Corollary \ref{cor_2}, there are no unbalanced distance-biregular graphs with girth deficiency two if $q\ge 1$. Therefore, we assume $q=K/L <1$ with $GCD(K,L)=1$. Set $k=mK$ and $l=mL$, for some positive integer $m$.

By multiplying both sides of the  congruence in Corollary \ref{cor_2} by $K^{r+1}$, we get $m^rK^rKL^r(mL+1)\equiv 0\,\mod\, (k+1)$, i.e., $m^rK^rL^r(mKL+K)\equiv 0\,\mod\, (k+1)$. Since $mK\equiv -1\,\mod\, (k+1)$, we get  $L^r(L-K)\equiv 0\,\mod\, (k+1)$. This yields that $k+1$ is a divisor of $L^r(L-K)$.
This implies that there are only finitely many possibilities for
$k+1$. For each such value, $l=kq^{-1}$ is fixed. Proposition
\ref{cor_1} then fixes $|V|=|A|+|B|$, and there are only finitely many
graphs on a fixed number of vertices.
\end{proof}
\section{The adjacency algebra of an unbalanced\\
distance-biregular graph with girth deficiency two} \label{sec_3}

Taking inspiration from \cite{hig}, in this section  we define an association scheme on the edge-set of an unbalanced distance-biregular graph with girth deficiency two. In order to define the (binary) relations {on the edges of the graph} we will make use of the {\em intersection diagram} with respect to an edge \cite{nom}. Although the definition can be given for any (connected) graph, for simplicity  we consider only distance-biregular graphs.

Let $\Gamma=(V,E)$ be a distance-biregular graph with partition classes $A$ and $B$.  Let $uv\in E$ with  $u\in A$ (hence $v\in B$). For $0\le i\le d_A$ and $0\le j\le d_B$, we define 
\[
D^i_j=D^i_j(u,v)=\Gamma_i(u)\cap\Gamma_j(v).
\]
Clearly $D^0_1=\{u\}$ and  $D^1_0=\{v\}$. In addition,  we have  $D^i_j=\varnothing$ if $|i-j|\ge2$ by the triangle inequality, and $D^i_i=\varnothing$, as $\Gamma$ is bipartite. 

The family
\[
\cD(u,v)=
\{D^i_{i+1}:0\le i\le d_B-1\}
\cup
\{D^{j+1}_j:0\le j\le\min\{d_A-1,d_B\}\},
\]
which consists of pairwise disjoint sets, is the
{\em intersection family with respect to the edge} $uv$, or the
$uv${\em-family} for short. The family
\[
\cD(u,v)=
\{D^i_{i+1}:0\le i\le d_B-1\}
\cup
\{D^{j+1}_j:0\le j\le\min\{d_A-1,d_B\}\},
\]
which consists of pairwise disjoint sets, is the
{\em intersection family with respect to the edge} $uv$, or the
$uv${\em-family} for short.It is useful to draw the $uv$-family in  a diagram as follows:
\captionsetup[figure]{name={Fig.}}
\begin{figure}[H]
\centering
\tikzstyle{vertexdraw}=[circle, draw=black, inner sep=1.5, minimum size=24.5pt]
\tikzstyle{vertexnodraw}=[circle, inner sep=1.5, minimum size=24.5pt]
\tikzstyle{vertexblack}=[circle, draw=black, fill=black, inner sep=0pt,minimum size=5pt]
\begin{tikzpicture}[baseline=-3]

   \draw (-0,0.75) node[vertexdraw] (D01) {$\scriptstyle{D^0_1}$};
   \draw (1.5,0.75) node[vertexnodraw] (dotsA1)  {$\scriptstyle{\cdots}$};
   \draw (3.0,0.75) node[vertexdraw] (Di-1i)  {$\scriptstyle{\scriptstyle{D^{i-1}_{i}}}$};
   \draw (4.5,0.75) node[vertexdraw] (Dii+1)  {$\scriptstyle{\scriptstyle{D^{i}_{i+1}}}$};
   \draw (6.0,0.75) node[vertexdraw] (Di+1i+2)  {$\scriptstyle{\scriptstyle{D^{i+1}_{i+2}}}$};
   \draw (7.5,0.75) node[vertexnodraw] (dotsA2)  {$\scriptstyle{\cdots}$};
   
   \draw (0,-0.75) node[vertexdraw] (D10) {$\scriptstyle{D^1_0}$};
   \draw (1.5,-0.75) node[vertexnodraw] (dotsB1)  {$\scriptstyle{\cdots}$};
   \draw (3.0,-0.75) node[vertexdraw] (Dii-1)  {$\scriptstyle{\scriptstyle{D^{i}_{i-1}}}$};
   \draw (4.5,-0.75) node[vertexdraw] (Di+1i)  {$\scriptstyle{\scriptstyle{D^{i+1}_{i}}}$};
   \draw (6.0,-0.75) node[vertexdraw] (Di+2i+1)  {$\scriptstyle{\scriptstyle{D^{i+2}_{i+1}}}$};
   \draw (7.5,-0.75) node[vertexnodraw] (dotsB2)  {$\scriptstyle{\cdots}$};

   \draw[-] (D01.east) -- (dotsA1.west);
   \draw[-] (dotsA1.east) -- (Di-1i.west);
   \draw[-] (Di-1i.east) -- (Dii+1.west);
   \draw[-] (Dii+1.east) -- (Di+1i+2.west);
   \draw[-] (Di+1i+2.east) -- (dotsA2.west);
   
   \draw[-] (D10.east) -- (dotsB1.west);
   \draw[-] (dotsB1.east) -- (Dii-1.west);
   \draw[-] (Dii-1.east) -- (Di+1i.west);
   \draw[-] (Di+1i.east) -- (Di+2i+1.west);
   \draw[-] (Di+2i+1.east) -- (dotsB2.west); 
   
   \draw[-] (D01.south) -- (D10.north);
      \draw[-] (Di-1i.south) -- (Dii-1.north);
   \draw[-] (Dii+1.south) -- (Di+1i.north);
   \draw[-] (Di+1i+2.south) -- (Di+2i+1.north); 
   \end{tikzpicture}
   \caption{The $uv$-diagram for a bipartite graph} \label{diag0}
\end{figure}
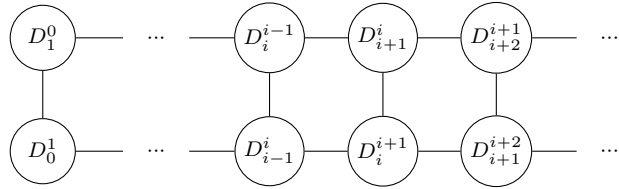

The above diagram  is the {\em intersection diagram with respect to the edge} $uv$, or the $uv${\em-diagram} for short; a line between two components indicates the possibility of the existence of an edge intersecting each of them. As $u\in A$, then $D^{2j}_{2j+1},D^{2j}_{2j-1}\subset A$, and  $D^{2j-1}_{2j},D^{2j+1}_{2j}\subset B$. Therefore, there is no edge inside each $D^i_j$. For $x\in D^{i}_{i+1}$, let
\[
a(x)=e(x,D^{i+1}_{i}) \ \ \ b(x)=e(x,D^{i+1}_{i+2}) \ \ \ c(x)=e(x,D^{i-1}_{i}),
\]
where  $e(x,Y)$ denotes the number of edges connecting $x$ with a vertex in the subset $Y$ of $V$. The triple $(a(x),b(x),c(x))$ is  the {\em edge pattern of} $x$ {\em with respect to} $uv$. The edge pattern of $x\in D^{i+1}_{i}$ with respect to $uv$ is defined symmetrically. The edge pattern of a vertex of a distance-biregular graph was calculated by Nomura in \cite{nom}.
\begin{lemma}\label{lem_20}\cite{nom}
Let  $\Gamma=(V,E)$  be a distance-biregular graph with intersection array
\[
\begin{array}{|c|}
\beta_0^A;\ \  1,\ \gamma_2^A,\ \cdots, \ \gamma_{d_A}^A\\[.05in]
\beta_0^B;\ \ 1,\ \gamma_2^B,\ \cdots, \ \gamma_{d_B}^B
\end{array}\ .
\]
Then,  there is the following relation between the edge pattern of $x\in V$ and the intersection numbers of $\Gamma$:
\[
(a(x),b(x),c(x))=\left\{
\begin{array}{ll}
(\beta^A_i-\beta^B_{i+1},\beta^B_{i+1},\gamma^A_{i}) & \mathrm{\ if \ } x\in D^{i}_{i+1}\\[.1in]
(\beta^B_i-\beta^A_{i+1},\beta^A_{i+1},\gamma^B_{i}) & \mathrm{\ if \ } x\in D^{i+1}_{i}
\end{array}
\right.,
\]
where, by convention, we set
$\gamma_0^A=\gamma_0^B=\beta_{d_A}^A=\beta_{d_B}^B=0$.
\end{lemma}
 We apply this result to our case.
\begin{proposition}
Let $\Gamma=(V,E)$ be an unbalanced distance-biregular graph with diameter $d$ and girth deficiency two. Then, for every  $uv\in E$, with $u\in A$, the $uv$-diagram is
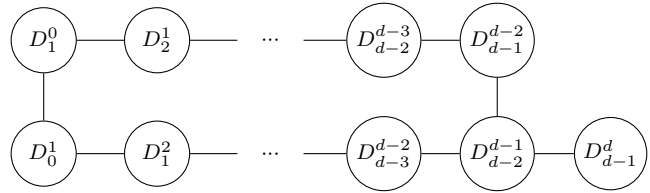
\begin{figure}[H]
\centering
\tikzstyle{vertexdraw}=[circle, draw=black, inner sep=1.5, minimum size=24.5pt]
\tikzstyle{vertex}=[circle, inner sep=1.5, minimum size=24.5pt]
\tikzstyle{vertexblack}=[circle, draw=black, fill=black, inner sep=0pt,minimum size=5pt]
\begin{tikzpicture}[baseline=-3]
  \draw (0,0.75) node[vertexdraw] (D01) {$\scriptstyle{D^0_1}$};
  \draw (1.5,0.75) node[vertexdraw] (D12)  {$\scriptstyle{\scriptstyle{D^{1}_{2}}}$};
   \draw (3.0,0.75) node[vertex] (dotsA1)  {$\scriptstyle{\cdots}$};
   
   \draw (4.5,0.75) node[vertexdraw] (Dd-3d-2)  {$\scriptstyle{\scriptstyle{D^{d-3}_{d-2}}}$};
   \draw (6.0,0.75) node[vertexdraw] (Dd-2d-1)  {$\scriptstyle{\scriptstyle{D^{d-2}_{d-1}}}$};

   \draw (0,-0.75) node[vertexdraw] (D10) {$\scriptstyle{D^1_0}$};   
   \draw (1.5,-0.75) node[vertexdraw] (D21)  {$\scriptstyle{\scriptstyle{D^{2}_{1}}}$};
   \draw (3.0,-0.75) node[vertex] (dotsB1)  {$\scriptstyle{\cdots}$};
   \draw (4.5,-0.75) node[vertexdraw] (Dd-2d-3)  {$\scriptstyle{\scriptstyle{D^{d-2}_{d-3}}}$};
   \draw (6.0,-0.75) node[vertexdraw] (Dd-1d-2)  {$\scriptstyle{\scriptstyle{D^{d-1}_{d-2}}}$};
   \draw (7.5,-0.75) node[vertexdraw] (Ddd-1)  {$\scriptstyle{\scriptstyle{D^{d}_{d-1}}}$};
   
   \draw[-] (D01.east) -- (D12.west);
   \draw[-] (D12.east) -- (dotsA1.west);
   \draw[-] (dotsA1.east) -- (Dd-3d-2.west);
   \draw[-] (Dd-3d-2.east) -- (Dd-2d-1.west);

   \draw[-] (D10.east) -- (D21.west);
   \draw[-] (D21.east) -- (dotsB1.west);
   \draw[-] (dotsB1.east) -- (Dd-2d-3.west);
   \draw[-] (Dd-2d-3.east) -- (Dd-1d-2.west);
   \draw[-] (Dd-1d-2.east) -- (Ddd-1.west);
      
   \draw[-] (D01.south) -- (D10.north);
   \draw[-] (Dd-2d-1.south) -- (Dd-1d-2.north);
 \end{tikzpicture}
   \caption{The $uv$-diagram for an unbalanced  distance-biregular graph with girth deficiency two} \label{diag1}
\end{figure}
\end{proposition}
\begin{proof}
Let $uv\in E$ and $x\in D^{i}_{i+1}$. {Recall that $d=2r+2$ and  $\gamma_0^A=\gamma_0^B=\beta_{d_A}^A=\beta_{d_B}^B=0$  by notational convention.}
By taking into account Lemma \ref{lem_20}, Proposition \ref{lem_2} and  Proposition \ref{cor_1}, Part 1, straightforward calculations give the following edge pattern of $x$ with respect to $uv$:
\[
\begin{array}{ll}
x\in D^{2j-1}_{2j}: & (a(x),b(x),c(x))=(0,l,1), \ \ \mathrm{for\ }j=1,\ldots, r;\\[.07in]
x\in D^{2j}_{2j+1}: & (a(x),b(x),c(x))=\left\{\begin{array}{ll}
(1,k,0) & \mathrm{for\ }j=0 \\[.05in]
(0,k,1) & \mathrm{for\ }j=1,\ldots, r-1 \\[.05in](k,0,1) &  \mathrm{for\ }j=r
\end{array}
\right.
\end{array}.
\]
Similarly, for $x\in D^{i+1}_{i}$ we have:
\[
\begin{array}{ll}
x\in D^{2j}_{2j-1}: & (a(x),b(x),c(x))=\left\{\begin{array}{ll}
(0,k,1) & \mathrm{for\ }j=1,\ldots, r \\[.05in]
(0,0,k+1) &  \mathrm{for\ }j=r+1
\end{array}
\right.;\\[.2in]
x\in D^{2j+1}_{2j}: & (a(x),b(x),c(x))=\left\{\begin{array}{ll}
(1,l,0) & \mathrm{for\ }j=0 \\[.05in]
(0,l,1) & \mathrm{for\ }j=1,\ldots, r-1 \\[.05in]
(k,l-k,1) &  \mathrm{for\ }j=r
\end{array}
\right.
\end{array}.
\]
\end{proof}

For any given pair $(uv,u'v')$ of edges of $\Gamma=(V,E)$, we adopt the convention that  $u\in A$  (while $u'$ is not necessarily) and $d(u',u)<d(v',u)$. We define the following non-diagonal relations on $E$, corresponding to the edges of the $uv$-diagram in Fig. \ref{diag1}:
\begin{equation}\label{eq_11}
\begin{array}{lll}
R^{\uparrow}_{i}: &
 (uv,u'v')\in R^{\uparrow}_{i}  \mathrm{\ if\ and\ only\ if\ } u'\in D^{i}_{i+1}(u,v), & 	\mathrm{for\ }i = 0,\ldots,d-2,\\[.1in]
R^{\downarrow}_{j}: &
 (uv,u'v')\in R^{\downarrow}_{j}  \mathrm{\ if\ and\ only\ if\ } u'\in D^{j+1}_{j}(u,v), & 	\mathrm{for\ }j = 0,\ldots, {d-2}.
\end{array}
\end{equation}

 In the next, it is useful to label the edges of the $uv$-diagram in Fig. \ref{diag1} with the relations $R^{\uparrow}_{i}$ and $R^{\downarrow}_{j}$ in the following way. 
 
 Let $(uv,u'v')\in R^{\uparrow}_{i}$. By taking into account the convention we adopted to write the elements in  $E\times E$, if $0\le i\le d-3$, then
$u'\in D^{i}_{i+1}$ and $v'\in D^{i+1}_{i+2}$, so the edge joining
$D^{i}_{i+1}$ and $D^{i+1}_{i+2}$ is labelled with
$R^{\uparrow}_{i}$; if $i=d-2$, then
$u'\in D^{d-2}_{d-1}$ and $v'\in D^{d-1}_{d-2}$, so the vertical edge
joining these two components is labelled with $R^{\uparrow}_{d-2}$. In a very similar way, the edge joining $D^{j+1}_{j}$ and $D^{j+2}_{j+1}$ is labelled with $R^{\downarrow}_{j}$. Thus, we obtain the diagram in Fig. \ref{diag3}.
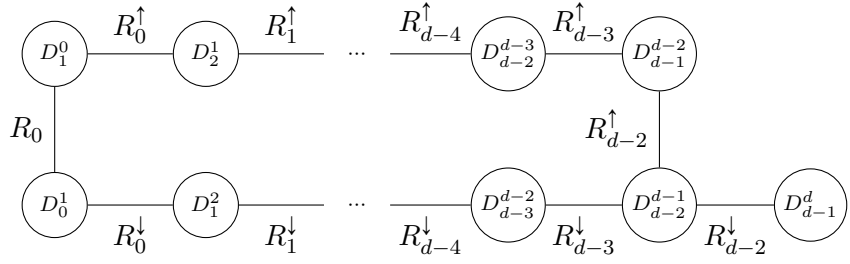
\begin{figure}[H]
\centering
\tikzstyle{vertexdraw}=[circle, draw=black, inner sep=1.5, minimum size=24.5pt]
\tikzstyle{vertex}=[circle, inner sep=1.5, minimum size=24.5pt]
\tikzstyle{vertexblack}=[circle, draw=black, fill=black, inner sep=0pt,minimum size=5pt]
\begin{tikzpicture}[baseline=-3]
  \draw (0-.5,0.75) node[vertexdraw] (D01) {$\scriptstyle{D^0_1}$};
  \draw (1.5,0.75) node[vertexdraw] (D12)  {$\scriptstyle{\scriptstyle{D^{1}_{2}}}$};
   \draw (3.0+.5,0.75) node[vertex] (dotsA1)  {$\scriptstyle{\cdots}$};
\draw (4.5+1,0.75)  node[vertexdraw] (D34)  {$\scriptstyle{\scriptstyle{D^{d-3}_{d-2}}}$};
\draw (6+1.5,0.75)  node[vertexdraw] (D45)  {$\scriptstyle{\scriptstyle{D^{d-2}_{d-1}}}$};
   \draw (0-.5,-1.25) node[vertexdraw] (D10) {$\scriptstyle{D^1_0}$};   
   \draw (1.5,-1.25) node[vertexdraw] (D21)  {$\scriptstyle{\scriptstyle{D^{2}_{1}}}$};
 \draw (3.0+.5,-1.25) node[vertex] (dotsB1)  {$\scriptstyle{\cdots}$};
\draw (4.5+1,-1.25)  node[vertexdraw] (D43)  {$\scriptstyle{\scriptstyle{D^{d-2}_{d-3}}}$};
\draw (6+1.5,-1.25)  node[vertexdraw] (D54)  {$\scriptstyle{\scriptstyle{D^{d-1}_{d-2}}}$};
\draw (7.5+2.0,-1.25)  node[vertexdraw] (D65)  {$\scriptstyle{\scriptstyle{D^{d}_{d-1}}}$};
   \draw[-] (D01.east) --node[font=\small,above]{$R_0^{\uparrow}$} (D12.west);
   \draw[-] (D12.east) --node[font=\small,above]{$R_1^{\uparrow}$} (dotsA1.west);
   \draw[-] (dotsA1.east) --node[font=\small,above]{$R_{d-4}^{\uparrow}$} (D34.west);
   \draw[-] (D34.east) --node[font=\small,above]{$R_{d-3}^{\uparrow}$} (D45.west);
   \draw[-] (D45.south) --node[font=\small,left]{$R_{d-2}^{\uparrow}$}  (D54.north);
   \draw[-] (D10.east) --node[font=\small,below]{$R^{\downarrow}_0$} (D21.west);
   \draw[-] (D21.east) --node[font=\small,below]{$R^{\downarrow}_1$} (dotsB1.west);
   \draw[-] (dotsB1.east) --node[font=\small,below]{$R^{\downarrow}_{d-4}$} (D43.west);
   \draw[-] (D43.east) --node[font=\small,below]{$R^{\downarrow}_{d-3}$} (D54.west);
   \draw[-] (D54.east) --node[font=\small,below]{$R^{\downarrow}_{d-2}$} (D65.west);       
   \draw[-] (D01.south) --node[font=\small,left]{$R_0$} (D10.north);
   \end{tikzpicture}
   \caption{The labelled $uv$-diagram of $\Gamma$.}\label{diag3}.
\end{figure}

As an example, the labelled $uv$-diagram for an unbalanced distance-biregular graph with diameter $d=6$ and girth deficiency two is shown below:

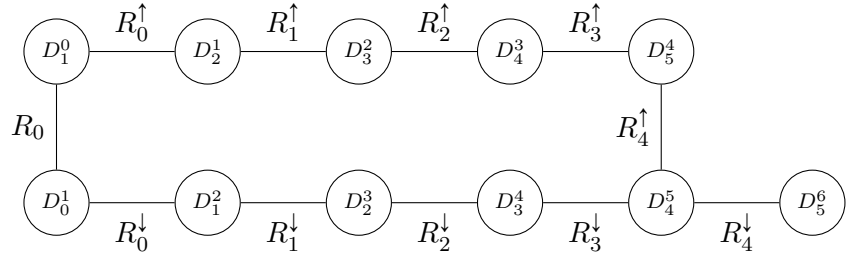
\begin{figure}[H]
\centering
\tikzstyle{vertexdraw}=[circle, draw=black, inner sep=1.5, minimum size=24.5pt]
\tikzstyle{vertex}=[circle, inner sep=1.5, minimum size=24.5pt]
\tikzstyle{vertexblack}=[circle, draw=black, fill=black, inner sep=0pt,minimum size=5pt]
\begin{tikzpicture}[baseline=-3]
  \draw (0-.5,0.75) node[vertexdraw] (D01) {$\scriptstyle{D^0_1}$};
  \draw (1.5,0.75) node[vertexdraw] (D12)  {$\scriptstyle{\scriptstyle{D^{1}_{2}}}$};
   \draw (3.0+.5,0.75) node[vertexdraw] (D23)  {$\scriptstyle{\scriptstyle{D^{2}_{3}}}$};
\draw (4.5+1,0.75)  node[vertexdraw] (D34)  {$\scriptstyle{\scriptstyle{D^{3}_{4}}}$};
\draw (6+1.5,0.75)  node[vertexdraw] (D45)  {$\scriptstyle{\scriptstyle{D^{4}_{5}}}$};
   \draw (0-.5,-1.25) node[vertexdraw] (D10) {$\scriptstyle{D^1_0}$};   
   \draw (1.5,-1.25) node[vertexdraw] (D21)  {$\scriptstyle{\scriptstyle{D^{2}_{1}}}$};
 \draw (3.0+.5,-1.25) node[vertexdraw] (D32)  {$\scriptstyle{\scriptstyle{D^{3}_{2}}}$};
\draw (4.5+1,-1.25)  node[vertexdraw] (D43)  {$\scriptstyle{\scriptstyle{D^{4}_{3}}}$};
\draw (6+1.5,-1.25)  node[vertexdraw] (D54)  {$\scriptstyle{\scriptstyle{D^{5}_{4}}}$};
\draw (7.5+2.0,-1.25)  node[vertexdraw] (D65)  {$\scriptstyle{\scriptstyle{D^{6}_{5}}}$};
   \draw[-] (D01.east) --node[font=\small,above]{$R_0^{\uparrow}$} (D12.west);
   \draw[-] (D12.east) --node[font=\small,above]{$R_1^{\uparrow}$} (D23.west);
   \draw[-] (D23.east) --node[font=\small,above]{$R_2^{\uparrow}$} (D34.west);
   \draw[-] (D34.east) --node[font=\small,above]{$R_3^{\uparrow}$} (D45.west);
   \draw[-] (D45.south) --node[font=\small,left]{$R_4^{\uparrow}$}  (D54.north);
   \draw[-] (D10.east) --node[font=\small,below]{$R^{\downarrow}_0$} (D21.west);
   \draw[-] (D21.east) --node[font=\small,below]{$R^{\downarrow}_1$} (D32.west);
   \draw[-] (D32.east) --node[font=\small,below]{$R^{\downarrow}_2$} (D43.west);
   \draw[-] (D43.east) --node[font=\small,below]{$R^{\downarrow}_3$} (D54.west);
   \draw[-] (D54.east) --node[font=\small,below]{$R^{\downarrow}_4$} (D65.west);       
   \draw[-] (D01.south) --node[font=\small,left]{$R_0$} (D10.north);
   \end{tikzpicture}
\end{figure}

For each $R^{\uparrow}_{i}$, $i=0,\ldots,d-2$, and
$R^{\downarrow}_{j}$, $j=0,\ldots,d-2$, we set
\begin{align*}
{R^{\uparrow}_{i}}'
  &=\{(wz,uv):(uv,wz)\in R^{\uparrow}_{i}\},\\
{R^{\downarrow}_{j}}'
  &=\{(wz,uv):(uv,wz)\in R^{\downarrow}_{j}\}.
\end{align*}
Here, according to our convention for representing elements of
$E\times E$, $w$ is the endpoint of $wz$ that belongs to $A$.
\begin{lemma} \label{lem_5} 
Let $\Gamma=(V,E)$ be an unbalanced distance-biregular graph with girth deficiency two. Then:
\begin{itemize}
\item[i)] ${R^{\uparrow}_{i}}'=\left\{\begin{array}{ll}
R^{\uparrow}_{i} & \mathrm{if\ } i \mathrm{\ is\ even,}\\[.05in]
R^{\downarrow}_{i} & \mathrm{if\ } i \mathrm{\ is\ odd.}
\end{array}\right.$
 
 \item[ii)] ${R^{\downarrow}_{i}}'=\left\{\begin{array}{ll}
R^{\downarrow}_{i} & \mathrm{if\ } i \mathrm{\ is\ even,}\\[.05in]
R^{\uparrow}_{i} & \mathrm{if\ } i \mathrm{\ is\ odd.}
\end{array}\right.$
  \end{itemize}
\end{lemma}
\begin{proof}
Let $(uv,u'v')\in R^{\uparrow}_{i}$, with $i\le d-3$ even.  Then $u'\in A$, with   $u'\in D^{i}_{i+1}(u,v)$ and $v'\in D^{i+1}_{i+2}(u,v)$ by applying the convention we adopted on $E\times E$. This yields that 
${R^{\uparrow}_{i}}'$ consists of the pairs $(u'v',uv)$, with $u\in D^{i}_{i+1}(u',v')$, that is   ${R^{\uparrow}_{i}}'=R^{\uparrow}_{i}$. Let $(uv,u'v')\in R^{\uparrow}_{i}$, with $i$ odd.  Then $v'\in A$, with   $v'\in D^{i+1}_{i+2}(u,v)$ and $u'\in D^{i}_{i+1}(u,v)$ by applying the convention we adopted on $E\times E$. This yields that  ${R^{\uparrow}_{i}}'$ consists of the pairs $(v'u',uv)$, with $u\in D^{i+1}_{i}(v',u')$, that is   ${R^{\uparrow}_{i}}'=R^{\downarrow}_{i}$. If $i=d-2$, then $u'\in A$,
$u'\in D^{d-2}_{d-1}(u,v)$ and
$v'\in D^{d-1}_{d-2}(u,v)$. Hence
$u\in D^{d-2}_{d-1}(u',v')$, and therefore
$(u'v',uv)\in R^{\uparrow}_{d-2}$. In the same manner we can prove that ${R^{\downarrow}_{i}}'=R^{\downarrow}_{i}$ if $i$ is even,  and ${R^{\downarrow}_{i}}'=R^{\uparrow}_{i}$ if $i$ is odd. 
 \end{proof}
Let  $A^\uparrow_{i}$,   $A^\downarrow_{j}$ be the adjacency matrices of the relations  $R^\uparrow_{i}$ and $R^\downarrow_{j}$, respectively. 

\begin{remark}\label{rem_16}
{\em By definition of $R^{\uparrow}_{0}$, the $(uv,u'v')$-entry of $A^{\uparrow}_{0}$ is non-zero if and only if we have $u'=u$ and $v'\neq v$. Similarly, by definition of $R^{\downarrow}_{0}$, the $(uv,u'v')$-entry of $A^{\downarrow}_{0}$ is non-zero if and only if we have  $u'=v$ and $v'\neq u$}.
\end{remark}

In the following, to make the notation easier, we set 
\[
 M=A^\uparrow_{0}, \ \ \ \  N=A^\downarrow_{0}.
\]
In the following, for given matrices $X$ and $Y$,  we write $\underbrace{XYX\ldots }_{\delta}$ to denote a product of $\delta$ matrices alternating between $X$ and $Y$. For example, $\underbrace{XYX\ldots }_{2}=XY$, $\underbrace{XYX\ldots }_{3}=XYX$ and $\underbrace{XYX\ldots }_{4}=XYXY$.

 {The products of the matrices $M$ and $N$ have a significant  combinatorial interpretation, which is described by the following result:}

\begin{lemma}\label{lem_4}
Let $\Gamma=(V,E)$ be a distance-biregular graph, and $(uv,u'v')\in E\times E$, with $u\in A$ and $d(u',u)<d(v',u)$. For any integer $\delta\ge 2$, 
\begin{itemize}
\item[1.] the $(uv,u'v')$-entry in the product  $\underbrace{MNM\cdots }_{\delta}$ is the number of paths of length $\delta +1$ starting with $v,u$ and ending with $u',v'$, namely,  paths of the form $v,u,w_1,\ldots, w_{\delta-2},u',v'$ in $\Gamma$;
\item[2.] the $(uv,u'v')$-entry in the product  $\underbrace{NMN\cdots }_{\delta}$ is the number of paths of length $\delta +1$ starting with $u,v$ and ending with $u',v'$, that is paths of the form $u,v,w_1,\ldots, w_{\delta-2},u',v'$ in $\Gamma$;
\end{itemize}
 here, $\delta$ is the total number of the matrices $M,N$ in the product.
\end{lemma}
\begin{proof}
 {Recall the  convention we adopted to represent elements in  $E\times E$: for any given pair $(uv,u'v')$ of edges of $\Gamma=(V,E)$, $u\in A$ (while $u'$ is not necessarily) and $d(u',u)<d(v',u)$.} We  proceed by induction on $\delta$. For $\delta=2$, by definition of the matrices $M$ and $N$,  and taking into account Remark \ref{rem_16}, the $(uv,u'v')$-entry in the product $MN$ is given by
\begin{align*}
(MN)_{(uv,u'v')}&=
\sum_{\substack{u''v''\in E }}{M_{(uv,u''v'')}N_{(u''v'',u'v')}}\\[.03in]
&=\sum_{\substack{ (u''u', u'v')\in R^{\downarrow}_{0} }}{M_{(uv,u''u')}},
\end{align*}
which is the number of the paths of type $v,u,u',v'$.

Let $\delta >2$, and assume that the $(uv,u''v'')$-entry of  {$\underbrace{MNM\cdots }_{\delta-1}$} is the number of paths of type $v,u,w_1,\ldots, w_{\delta-3},u'',v''$ in $\Gamma$. 
\\\indent
For the sake of simplicity assume that $\delta$ is even. The $(uv,u'v')$-entry in the product $\underbrace{MN\cdots N}_\delta$ is given by
\[
({MN\cdots }N)_{(uv,u'v')}=
\sum_{\substack{ (u''u', u'v')\in R^{\downarrow}_{0} }}{({MN\cdots })_{(uv,u''u')}},
\]
which is the number of the paths of type $v,u,w_1,\ldots, w_{\delta-2},u',v'$. 

Similar arguments apply to the case of odd $\delta$ and of the product  {$\underbrace{NMN\cdots}_\delta$}.
\end{proof}

As a corollary we get the following: 
\begin{proposition}\label{rem_11}
If  $\delta\le d-1$, then $\underbrace{MNMN\cdots }_{\delta}$ and $\underbrace{NMNM\cdots }_{\delta}$ are $\{0,1\}$-matrices.
\end{proposition}
\begin{proof}
Consider the product $\underbrace{MNMN\cdots }_{\delta}$, with $\delta\le d-1$. By Lemma \ref{lem_4}, Part 1,  the $(uv,u'v')$-entry in  {$\underbrace{MNM\cdots }_{\delta}$} is the number of paths  of length $\delta +2$  of the form $v,u,w_1,\ldots, w_{\delta-2},u',v'$ in $\Gamma$. 
\\
Let $v,u,w_1,\ldots,w_{\delta-2},u',v'$ and $v,u,w'_1,\ldots,w'_{\delta-2},u',v'$ be two distinct paths of the given form. Then  the walk $u,w_1,\ldots,w_{\delta-2},u',w'_{\delta-2},\ldots,w'_1,u$ contains a cycle of length at most $2\delta-2\le 2d-4$, which is against the assumption that $\Gamma$ has  girth  $2d-2$. Very similar arguments apply to the product $\underbrace{NMN\cdots }_{\delta}$, with  $\delta\le d-1$.  
\end{proof}

 Let $\cM_\Gamma$ be the subset of $\Mat_E(\bC)$ consisting of  {the} identity matrix $I$ and the adjacency  matrices $A^\uparrow_{i}$ and   $A^\downarrow_{i}$, for $i= 0,\ldots,d-2$; we recall that $\Mat_E(\bC)$ is the standard  {matrix algebra over} $\bC$ whose rows  {and} columns are indexed by the elements of $E$. 
\begin{proposition}\label{prop_4}
Let $\Gamma=(V,E)$ be an unbalanced distance-biregular graph with diameter $d$, girth deficiency two and partition classes $A$ and $B$. Then, as an algebra, $\Span_{\bC}(\cM_\Gamma)$ is generated by $I$, $M$ and $N$. 
\end{proposition}
\begin{proof}
We claim that $M^2=(k-1)M+kI$. Indeed, for each $(uv,u'v')\in E\times E$ the $(uv,u'v')$-entry of $M^2$ corresponds to the number of edges $u''v''$ such that $(uv, u''v'')\in R^{\uparrow}_{0}$ and $(u''v'', u'v')\in R^{\uparrow}_{0}$. By the definition of $R^{\uparrow}_{0}$ and taking into account Remark \ref{rem_16}, if this entry is non-zero we have $u=u''=u'$ and $v\neq v''\neq v'$. There are two cases to consider: $v=v'$ and $v\neq v'$. If $v=v'$ then  $uv$ coincides with $u'v'$. Since $u$ has $k+1$ neighbors, then there are $k$ possible choices for $v''$. If $v\neq v'$ then $(uv, u'v')\in R^{\uparrow}_{0}$ and there are $k-1$ possible choices for $v''$. Similarly,  for each $(uv,u'v')\in E\times E$ the $(uv,u'v')$-entry of $N^2$ corresponds to the number of edges $u''v''$ such that $(uv, u''v'')\in R^{\downarrow}_{0}$ and $(v''u'', u'v')\in R^{\downarrow}_{0}$ (note that $v''\in A$). By the definition of $R^{\downarrow}_{0}$, if this entry is non-zero we have $v=u''=u'$ and $u\neq v''\neq v'$. There are two cases to consider: $u=u'$ and $u\neq u'$. If $u=u'$ then $uv$ coincides with $u'v'$. Since $v$ has $l+1$ neighbors, then there are $l$ possible choices for $v''$. If $u\neq u'$ then $(uv, u'v')\in R^{\downarrow}_{0}$ and there are $l-1$ possible choices for $v''$.

This yields $N^2=(l-1)N+lI$. Therefore, the algebra generated by $M$ and $N$ is spanned by the products $MNMN\cdots$ and $NMNM\cdots$.

According to Proposition \ref{rem_11}, the product $\underbrace{MN\cdots}_{\delta}$ is a $(0,1)$-matrix, for $1\le\delta\le d-1$. Consider $(uv,u'v')\in E\times E$ such that the $(uv,u'v')$-entry of $\underbrace{MN\cdots}_{\delta}$ is non-zero, that is, by Lemma \ref{lem_4}, there is a path of type $v,u,w_1,\ldots, w_{\delta-2},u',v'$. Using Proposition \ref{rem_11}, we see that $u,w_1,\ldots, w_{\delta-2},u'$ is the unique path between $u$ and $u'$, and this shows that $u'\in D^{\delta-1}_\delta$. Looking at the $uv$-diagram of Fig. \ref{diag3}, we get $(uv,u'v')\in R^\uparrow_{\delta-1}$. Similarly, if the $(uv,u'v')$-entry of $\underbrace{NM\cdots}_{\delta}$ is non-zero, we can show that $u'\in D^\delta_{\delta-1}$.  By taking into account the $uv$-diagram in Fig. \ref{diag3}, we see that $v'\in D^{\delta+1}_\delta$ whenever $\delta\leq d-2$. If $\delta=d-1$, we can have either $v'\in D^{d-2}_{d-1}$ or $v'\in D^d_{d-1}$. Therefore we have that
\begin{equation}\label{eq_9}
\begin{array}{ll}
\underbrace{MN\cdots }_{\delta}=A^\uparrow_{\delta-1}, & \mathrm{for\ } \delta=1,\ldots,d-1,\\[.1in]
\underbrace{NM\cdots}_{\delta}=A^\downarrow_{\delta-1},& \mathrm{for\ } \delta=1,\ldots,d-2, \\[.1in]
\underbrace{NM\cdots N}_{d-1}=A^\uparrow_{d-2}+A^\downarrow_{d-2}.
\end{array}
\end{equation}

To simplify notation, we set $D=A^\downarrow_{d-2}$. Then, from Eq. \eqref{eq_9} we get 
\begin{equation}\label{eq_34}
	D=\underbrace{NM\cdots MN}_{d-1}-\underbrace{MN\cdots NM}_{d-1}.	
\end{equation}

We claim that $MD=DM=kD$. Let $uv,u'v'\in E$ such that the $(uv,u'v')$-entry of $MD$ is non-zero. Then, there exists  $uv''\in E$ such that $(uv,uv'')\in R^\uparrow_{0}$  and $(uv'',u'v')\in R^\downarrow_{d-2}$.  As $d(u',u)=d-1$ and $d(v',u)=d$,  by taking into account the $uv$-diagram in Fig. \ref{diag3} we see that $(uv,u'v')\in R^\downarrow_{d-2}$. Moreover, $v''$ can be chosen in $k$ different ways, one for each neighbor of $u$ different from $v$. So, we get $(MD)=kD$.
In addition, Lemma \ref{lem_5} implies  $DM=D^\top M^\top =(MD)^\top =kD^\top =kD$.

 {Recalling that $M^2=(k-1)M+kI$, from Equation \eqref{eq_34} we get
\begin{align*}
MD
&=\underbrace{MN\cdots MN}_{d}
  -M^2\underbrace{N\cdots NM}_{d-2}\\
&=\underbrace{MN\cdots MN}_{d}
  -(k-1)\underbrace{MN\cdots NM}_{d-1}\\
&\qquad-k\underbrace{N\cdots NM}_{d-2}.
\end{align*}

On the other hand, we just proved that $MD=kD$, so
\[ MD = kD = k\underbrace{NM\cdots MN}_{d-1}-k\underbrace{MN\cdots NM}_{d-1}. \]
By comparing the above expression for $MD$, we get}
\[
\underbrace{MN\cdots MN}_{d}= k\underbrace{NM\cdots MN}_{d-1} - \underbrace{MN\cdots NM}_{d-1} + k\underbrace{NM\cdots NM}_{d-2},
\] 
and, taking the transpose, we obtain
\[
\underbrace{NM\cdots NM}_{d}= k\underbrace{NM\cdots MN}_{d-1} - \underbrace{MN\cdots NM}_{d-1} + k\underbrace{MN\cdots MN}_{d-2}.
\] 
Therefore, every time we have a product of $M$ and $N$ with  at least $d$ factors, we can express it as a linear combination of products with strictly fewer factors. By Eq. \eqref{eq_9}, the proof is complete.
\end{proof}
\begin{corollary}\label{rem_3a}
Let $\Gamma=(V,E)$ be an unbalanced distance-biregular graph with diameter $d=2r+2$ and girth deficiency two. Then the adjacency matrices of the relations \eqref{eq_11} are as follows:
\[
\begin{array}{ll}
\begin{array}{rcl}
A^\uparrow_{2i-1}& = & (MN)^{i}, \\[.05in]
A^\uparrow_{2i}  & = &    (MN)^{i}M, 
\end{array}
& \mathrm{for\ }  i=1,\ldots,r, \\[.25in]
\begin{array}{rcl}
A^\downarrow_{2i-1}& = & (NM)^{i},\\[.05in]
A^\downarrow_{2i}  & = &    (NM)^{i}N,
\end{array}
& \mathrm{for\ } i=1,\ldots,r-1,\\[.25in]
\begin{array}{lcl}
A^\downarrow_{d-3} & = & (NM)^r, \\[.05in]
A^\downarrow_{d-2}  & = & (NM)^rN- (MN)^rM.
\end{array}
\end{array}
\]
\end{corollary}
\begin{proof}
The equalities in the statement are obtained from the proof of Proposition \ref{prop_4} by writing the adjacency matrices of the relations \eqref{eq_11} in terms of the matrices $M=A^\uparrow_{0}$ and $N=A^\downarrow_{0}$. 
\end{proof}
\begin{remark}\label{rem_3}{\em 
 By using the equalities given in Corollary \ref{rem_3a}, the valency $\eta_s$ of the relation $s$ is easily computed via the formula  $\eta_s w=A_sw$, where $A_s$ is the adjacency matrix of $s$ and $w$ is the all-ones vector.
 }
\end{remark}
For any unbalanced distance-biregular graph $\Gamma=(V,E)$ with diameter $d$, girth deficiency two and vertex degrees $k+1$ and $l+1$, we will denote the subalgebra $\Span_\bC(\cM_\Gamma)$ by $\cA_{d}(k,l)$.
\begin{corollary}\label{cor_7}
 {Let $\Gamma=(V,E)$ be an unbalanced distance-biregular graph with diameter $d=2r+2$ and girth deficiency two. Then, the algebra }$\cA_{d}(k,l)$ is an adjacency algebra of dimension $2d-1$. Therefore, $\cX=( E,\cR)$, where $\cR$ is the set of the relations described in \eqref{eq_11}, is an association scheme.
\end{corollary}
\begin{proof}
Since the matrices described in  {Corollary} \ref{rem_3a} form a linear basis for $\cA_{d}(k,l)$,  the dimension of $\cA_{d}(k,l)$ is $2d-1$. 
Theorem \ref{th_2} together with Lemma \ref{lem_5} and Proposition \ref{prop_4} imply that $\cA_{d}(k,l)$ is the adjacency algebra of the scheme $\cX=( E,\cR)$.
\end{proof}
{
\begin{remark}\label{rem_15}
{\em Let $\Gamma=(V,E)$ be an  unbalanced distance-biregular graph with diameter $d=2r+2$ and girth deficiency two, and $\cA_{d}(k,l)$ its adjacency algebra. In terms of the language of algebra, $\cA_{d}(k,l)$ can be viewed as the algebra generated by elements $M$ and $N$, and defined by  the following  identities:
\begin{equation} \label{eq_16}
M^2=(k-1)M+kI,
\end{equation}
\begin{equation} \label{eq_17}
N^2=(l-1)N+lI,
\end{equation}
\begin{equation} \label{eq_18}
MD=kD=DM,
\end{equation}
where $D=(NM)^rN-(MN)^rM$. } 
\end{remark}
}
\section{The irreducible representations of $\cA_{d}(k,l)$}\label{sec_4}

Let $\Gamma=(V,E)$ be an unbalanced distance-biregular graph with diameter $d=2r+2$ and girth deficiency two, and let $\cA_{d}(k,l)$ denote its adjacency algebra. In this section, we compute the irreducible representations of $\cA_{d}(k,l)$, considering this algebra generated by $M$ and $N$ as described in Remark \ref{rem_15}. This implies that, for any vector space $U$ and homomorphism
$\rho:\cA_{d}(k,l)\rightarrow\End(U)$, the pair $(U,\rho)$ is a
representation of $\cA_{d}(k,l)$ if and only if $\rho(M)$ and
$\rho(N)$ satisfy Eqs. \eqref{eq_16}, \eqref{eq_17} and
\eqref{eq_18}.
 We recall that $(U,\rho)$ is an irreducible representation of $\cA_{d}(k,l)$  if $U$  does not contain non-trivial submodules. {In the following, we  refer to the equations \eqref{eq_16}, \eqref{eq_17} and \eqref{eq_18}  as the {\em defining equations} of the algebra $\cA_{d}(k,l)$. We will consider them as matrix equations even though the matrices to which we apply them will not necessarily be $M$ and $N$.}

\subsection{The 1-dimensional irreducible representations of $\cA_{d}(k,l)$}

\begin{proposition}\label{prop_3}
Up to isomorphism, $\cA_{d}(k,l)$ has exactly three 1-dimensional representations, denoted by $U_{1,1}$, $U_{1,2}$ and $U_{1,3}$. Their values are summarized in Table \ref{tab_2}, where $U_{1,1}$ is the principal representation and $D=(NM)^rN-(MN)^rM$:  
{\renewcommand\arraystretch{1.3}
\begin{center}
\captionof{table}{The 1-dimensional representations of $\cA_{d}(k,l)$}
\label{tab1}
\begin{tabular}{c|c|c|c|c|c}
\hline
& \parbox[c][.4in][c]{0.8in}{\centering
  \fontsize{11pt}{0pt}\selectfont{$(MN)^i$}\\[.04in]
  \fontsize{8pt}{0pt}\selectfont{$i=0,\ldots,r$}}
& \parbox[c][.4in][c]{.8in}{\centering
  \fontsize{11pt}{0pt}\selectfont{$(NM)^i$}\\[.04in]
  \fontsize{8pt}{0pt}\selectfont{$i=1,\ldots,r$}}
& \parbox[c][.6in][c]{.8in}{\centering
  \fontsize{11pt}{0pt}\selectfont{$(MN)^iM$}\\[.04in]
  \fontsize{8pt}{0pt}\selectfont{for $i=0,\ldots,r$}}
& \parbox[c][.6in][c]{.95in}{\centering
  \fontsize{11pt}{0pt}\selectfont{$(NM)^iN$}\\[.04in]
  \fontsize{8pt}{0pt}\selectfont{$i=0,\ldots,r-1$}}
& \parbox[c][.6in][c]{.5in}{\centering
  \fontsize{11pt}{0pt}\selectfont{$D$}}
\\[.04in]
\hline
$U_{1,1}$ &
$k^il^i$ &
$k^il^i$ &
$k^{i+1}l^i$ &
$k^il^{i+1}$ &
$k^rl^r(l-k)$
\\[.08in]
$U_{1,2}$ &
$1$ &
$1$ &
$-1$ &
$-1$ &
$0$
\\[.08in]
$U_{1,3}$ &
$(-1)^ik^i$ &
$(-1)^ik^i$ &
$(-1)^ik^{i+1}$ &
$(-1)^{i+1}k^i$ &
$(-1)^{r+1}k^r(k+1)$
\\
\hline
\end{tabular}
\label{tab_2}
\end{center}
}
\end{proposition}
\begin{proof}
Let $(\bC,\rho)$ be a 1-dimensional representation of $\cA_{d}(k,l)$.  From Eq.\eqref{eq_16}, it follows that $\rho(M)^2=\rho(M^2)=(k-1)\rho(M)+k$. Therefore, $\rho(M)\in\{-1,k\}$.
Similarly, Eq.\eqref{eq_17} yields $\rho(N)\in\{-1,l\}$.

From Eq.\eqref{eq_18}, we have either  $\rho(D)=0$ or $\rho(M)=k$. This excludes the case $\rho(M)=-1$ and $\rho(N)=l$.
Therefore, we have exactly three 1-dimensional representations of $\cA_d(k,l)$.  The principal representation is the one with $\rho(M)=k$ and $\rho(N)=l$.
\end{proof}

\begin{remark}
 {
{\em 
Let $\Gamma=(V,E)$ be an unbalanced distance-biregular graph with diameter $d=2r+2$ and girth deficiency two. By Remark \ref{rem_3}, the valency of the relations $R^\uparrow_i$ and $R^\downarrow_i$, for $i=0,\ldots, d-2$, is computed  via the formula  $\eta_s w=A_sw$, where $A_s$ is the adjacency matrix of the relation $s$ and $w$ is the all-ones vector. Since $U=\<w\>$ is the principal representation of $\cA_{d}(k,l)$, whose associated homomorphism is $\rho(A_s)=\eta_s$ for $s\in\cR$, we may use the values of the principal character reported in Table \ref{tab1} to get the valency of each relation. In addition we have 
\begin{align*}
|E| & = \sum_{s\in \cR}\eta_s = 1+2\sum_{i=1}^{r}{k^il^i} +\sum_{i=0}^{r}{k^{i+1}l^i}+\sum_{i=0}^{r-1}{k^{i}l^{i+1}}+k^rl^r(l-k)\\[0.08in]
& =  \frac{(k^{r+1}l^{r}(l+1)-k-1)(l+1)}{kl-1},
\end{align*}
that was already found in Corollary \ref{cor_1}.
}
}
\end{remark}

\subsection{The 2-dimensional irreducible representations of $\cA_{d}(k,l)$}
In this subsection, we study the  2-dimensional irreducible representations of $\cA_{d}(k,l)$. We will find that such representations exist and are in 1-1 correspondence with the zeros of a suitable integer polynomial of degree $r$. 

As a first step, we give a useful property  of the two-dimensional irreducible representations of $\cA_{d}(k,l)$. 
\begin{proposition}\label{lem_21}
Let $(U,\rho)$ be a 2-dimensional irreducible representation of $\cA_{d}(k,l)$. Then, up to isomorphism, we have
	\[
	\rho(M)=\begin{pmatrix} k & 0 \\ 0 & -1
	\end{pmatrix}\ \ \ \ \mathrm{ and\ }\ \  \rho(N)=\begin{pmatrix} 
		x & z \\ z & y \end{pmatrix},
		\]
	where  $t=\Trace(\rho(MN))$, and $x$, $y$ and $z$ satify 
	\[
	t=kx-y
	\]
	 and
	\[
x  = \frac{t+l-1}{k+1},\ \ \ y=\frac{kl-k-t}{k+1},\ \ \ z^2=\frac{(1+kl-t)(t+k+l)}{(k+1)^2}.
\]
\end{proposition}
\begin{proof}
From Eq.\eqref{eq_16} we see that the minimal polynomial of $\rho(M)$ must divide the polynomial $X^2-(k-1)X-k=(X-k)(X+1)$. This yields that $\rho(M)$ is diagonalizable, with eigenvalues in $\{k,-1\}$. Similarly, $\rho(N)$ is diagonalizable with eigenvalues in $\{l,-1\}$ by Eq.\eqref{eq_17}. If $\rho(M)$ is a scalar matrix, then any eigenvector of $\rho(N)$ is an  eigenvector of $\rho(M)$, whence  $U$ would have a non-trivial subrepresentation, against the irreducibility of $U$. 
Hence, $\rho(M)$ has two distinct eigenvalues, namely $k$ and $-1$. Similarly, the eigenvalues of  $\rho(N)$ are $l$ and $-1$.

Let's choose  two orthonormal eigenvectors of $\rho(M)$ as the basis of $U$. With respect to this basis  we have $\rho(M)=\begin{pmatrix} k & 0 \\ 0 & -1 \end{pmatrix}$. Let $\rho(N)=\begin{pmatrix} x & z \\ w & y \end{pmatrix}$. With this  choice of basis, $\rho(A^*)=\rho(A)^*$ holds for any $A\in\cA_{d}(k,l)$, where  $A^*$ is the  Hermitian transpose of $A$.   Therefore, $\rho(N)$ is  hermitian because $N$ is. This yields $\rho(N)=\begin{pmatrix} x & z \\ \bar z & y \end{pmatrix}$, with $z\neq 0$.

By Eq.\eqref{eq_17}, we see that $x$, $y$ and $z$ are complex numbers such that
\[
	\left\{
	\begin{array}{lclcl}
		x+y & = & \Trace(\rho(N)) & = & l-1\\[.03in]
		xy-z\bar z& = & \det(\rho(N)) & = &  -l
	\end{array}\right..
	\] 
\\
Set $t=\Trace(\rho(MN))=kx-y$. It is easy to check that 
\[
x  = \frac{t+l-1}{k+1},\ \ \ y=\frac{kl-k-t}{k+1},\ \ \ z\bar z=\frac{(1+kl-t)(t+k+l)}{(k+1)^2}.
\]
We note that  the change of basis in $U$ defined by the matrix  $\begin{pmatrix}
		1 & 0 \\
		0 & e^{i\theta}
	\end{pmatrix}$ leaves $\rho(M)$ invariant and maps  $\rho(N)$  to $\begin{pmatrix}
		x & ze^{i\theta} \\
		\bar ze^{-i\theta} & y
	\end{pmatrix}$.  Hence we can assume that $z=\bar z$ and this completes the proof.
\end{proof}
Let $S$ be  the  $2\times 2$ matrix
	\[
		S=\begin{pmatrix} k & 0 \\ 0 & -1
	\end{pmatrix}
\]
and, for any $t\in\bC$, let $T_t$ be the  $2\times 2$ matrix
\begin{equation}\label{eq_5}
 T_t=\begin{pmatrix} 
		x(t) & z(t) \\ z(t) & y(t) 
		\end{pmatrix},
\end{equation}
where $x(t)$, $y(t)$ and $z(t)$ are given by 
\begin{equation}\label{eq_8}
x(t)  = \frac{t+l-1}{k+1},\ \ \ y(t)=\frac{kl-k-t}{k+1},\ \ \ z^2(t)=\frac{(1+kl-t)(t+k+l)}{(k+1)^2}.
\end{equation}
Straightforward computations show that $t=kx(t)-y(t)=\Trace(ST_t)$.
 For  any $t\in\bC$, we consider the matrices $S$ and $T_t$.  Whenever $S$ and $T_t$ satisfy Eqs. \eqref{eq_16}, \eqref{eq_17} and  \eqref{eq_18}, there exists a homomorphism  $\rho_t:\cA_{d}(k,l)\rightarrow \Mat_2(\bC)$, such that $(\bC^2,\rho_t)$ is a 2-dimensional representation of $\cA_{d}(k,l)$.
 From Proposition \ref{lem_21}, we see that any 2-dimensional irreducible representation  of $\cA_{d}(k,l)$ is of this type. 
 
In what follows, to lighten the notation and when this does not create confusion, we will use $T=\begin{pmatrix} 
		x & z \\ z & y
		\end{pmatrix}$  instead of 
$T_t=\begin{pmatrix} 
		x(t) & z(t) \\ z(t) & y(t) 
		\end{pmatrix}$. With this notation we have 
		\begin{equation}\label{eq_43}
		t=\Trace(ST)=kx-y.
		\end{equation}

\begin{lemma}\label{lem_23}
For any $t\in\bC$, the matrices $S$ and $T$ satisfy Eqs. \eqref{eq_16} and \eqref{eq_17}.
\end{lemma}
\begin{proof}
By direct computation, we see that $S^2=(k-1)S+kI$. Since $\Trace(T)=l-1$ and $\det(T)=-l$, we have  $T^2=(l-1)T+lI$ by the Cayley-Hamilton Theorem. 
\end{proof}
 The main tool we use to find the 2-dimensional representations of
$\cA_d(k,l)$ is the theory of second-order linear recurrence equations
with constant coefficients.

A {\em linear recurrence equation of order 2 with constant coefficients}  is a recurrence equation of type  
\begin{equation}\label{eq_39}
X_{i+2}+\alpha_1X_{i+1} + \alpha_0X_{i} = 0, \qquad i=0,1,\ldots,
\end{equation}
where $\alpha_1,\alpha_0\in\bC$, with $\alpha_0\neq 0$. Let $V$ be any vector space over $\bC$. A sequence $\{v_i\}_{i=0}^\infty$ of elements of $V$ is a {\em solution} of \eqref{eq_39} if $v_{i+2}+\alpha_1v_{i+1} + \alpha_0v_{i} = 0$, for  $i=0,1,\ldots$.

In the next result we collect the main properties of linear recurrence equations of order 2 with constant coefficients. For a general discussion of this topic and proofs of individual results, we refer the reader to \cite[Section 4]{kp}.

\begin{proposition}\label{prop_11}
Let 
\begin{equation}\label{eq_36}
	X_{i+2}+\alpha_1X_{i+1} + \alpha_0X_{i} = 0, \qquad i=0,1,\ldots,
\end{equation}
be a linear recurrence equation of order 2 with constant coefficients $\alpha_0,\alpha_1\in\bC$, $\alpha_0\neq 0$, and let $V$ be a vector space over $\bC$. Then the following hold:
\begin{itemize}
\item[i.] If $\{v_i\}_{i=0}^\infty$ and $\{w_i\}_{i=0}^\infty$ are  solutions of  \eqref{eq_36} with $v_0=w_0$ and $v_1=w_1$, then $v_i=w_i$ for  $i=0,1,\ldots$.
\item[ii.] If the polynomial $X^2+\alpha_1X+\alpha_0$ has two distinct
roots in $\bC$, say $\lambda$ and $\mu$, then every solution
$\{v_i\}_{i=0}^\infty$ of \eqref{eq_36} can be written as $v_i=\lambda^i v+\mu^i w$, $i=0,1,\ldots$, for some uniquely determined $v,w\in V$  depending only on $v_0$ and $v_1$.
\item[iii.] If $\{v_i\}_{i=0}^\infty$ is a solution of \eqref{eq_36} in $V$  and $\phi:V\to W$ is a linear map from $V$ to the vector space $W$, then $\{\phi(v_i)\}_{i=0}^\infty$ is a solution of \eqref{eq_36} in $W$.
\end{itemize}
\end{proposition}

 We use this machinery to determine the properties of the matrices found in Proposition \ref{lem_21}.

\begin{proposition}\label{prop_10}
	Let $t\in\bC$ and   set
\begin{equation}\label{eq_41}
S=\begin{pmatrix} k & 0 \\ 0 & -1
	\end{pmatrix},\ \ \ T=\begin{pmatrix} 
		x & z \\ z & y \end{pmatrix},
\end{equation}
where $x=x(t)$, $y=y(t)$ and $z=z(t)$ are given by \eqref{eq_8}. Then,  the sequences $\{(ST)^i\}_{i=0}^\infty$, $\{(TS)^i\}_{i=0}^\infty$, $\{(TS)^iT\}_{i=0}^\infty$, $\{(TS)^iT-(ST)^iS\}_{i=0}^\infty$, and the sequences of the corresponding traces are solutions of the linear recurrence equation \begin{equation}\label{eq_13}
			X_{i+2} - tX_{i+1} + klX_{i}=0, \qquad i=0,1,\ldots.
	\end{equation}
\end{proposition}
\begin{proof}
	By direct computation, we have $\det(ST)=kl$ and $\Trace(ST)=t$. Therefore, the characteristic polynomial of $ST$ is $X^2-tX+kl$. According to the Cayley-Hamilton theorem, the matrix $ST$ satisfies its own characteristic equation, that is $(ST)^2-tST+klI=0$. Multiplying this equation by $(ST)^i$, we get $(ST)^{i+2}-t(ST)^{i+1}+kl(ST)^i=0$, for $i=0,1,\ldots$. Therefore, the sequence $\{(ST)^i\}_{i=0}^\infty$ satisfies the linear recurrence equation \eqref{eq_13}. 
	
Since transposition is a linear map on $\Mat_2(\bC)$ and $TS=(ST)^\top $, we get that $\{(TS)^i\}_{i=0}^\infty$ is a solution of \eqref{eq_13} by property iii. in Proposition \ref{prop_11}. 
Similarly, by considering right multiplication by $S$ or $T$ as a
linear map on $\Mat_2(\bC)$, we get that
$\{(TS)^iT\}_{i=0}^\infty$ and
$\{(ST)^iS\}_{i=0}^\infty$ are solutions of \eqref{eq_13}.
The same holds for
$\{(TS)^iT-(ST)^iS\}_{i=0}^\infty$, since it is a linear combination
of solutions of \eqref{eq_13}.
 Finally, as  $\Trace:\Mat_2(\bC)\to \bC$ is a linear map,  the sequences given by the traces of $(ST)^i$, $(TS)^i$, $(TS)^iT$ and $(TS)^iT-(ST)^iS$, for $i=0,1,\ldots$, are solutions of \eqref{eq_13}.
\end{proof}

In the following, to make the notation shorter, we set 
\[
D_i=(TS)^iT-(ST)^iS,\qquad i=0,1,\ldots.
\]
 Recall that in Remark \ref{rem_15} we set $D=(NM)^rN-(MN)^rM$. So, according to Eq. \eqref{eq_18} in that remark,  we have to find conditions on $t\in\bC$ such that 
\begin{equation}\label{eq_40}
SD_r=kD_r=D_rS
\end{equation}
holds.  For this aim,  we need to compute the entries of $D_i$, for $i=0,1,\ldots$. We do this in the next result.
\begin{lemma}\label{lem_15}
Let $t\in\bC$, $S$ and $T$ as in \eqref{eq_41}, and $D_i=(TS)^iT-(ST)^iS$, for $i=0,1,\ldots$.
Then, there exist solutions $\{f_i\}_{i=0}^\infty$ and $\{h_i\}_{i=0}^\infty$ of the linear recurrence equation 
\begin{equation}\label{eq_13bis}
			X_{i+2} - tX_{i+1} + klX_{i}=0, \qquad i=0,1,\ldots,
	\end{equation}
 such that 
\[
D_i=\begin{pmatrix} h_i & zf_i \\ zf_i & (y+1)f_i\end{pmatrix}, \qquad i=0,1,\dots.
\]
 In particular, $f_0=1$, $f_1=t+k$, and $f_i=tf_{i-1}-klf_{i-2}$ for $i=2,3,\ldots.$
\end{lemma}
\begin{proof}
We note that $D_i=(TS)^iT-(ST)^iS$, $i=0,1,\dots$, is a symmetric matrix because $S$ and $T$ are. Set $D_i =\begin{pmatrix} h_i & \hat f_i \\ \hat f_i & g_i\end{pmatrix}$, for  $i=0,1,\dots$. By Proposition \ref{prop_10}, the sequence $\{D_i\}_{i=0}^\infty$ is a solution of the recurrence equation \eqref{eq_13bis}. This yields that the  sequences $\{h_i\}_{i=0}^\infty$,  $\{\hat f_i\}_{i=0}^\infty$ and $\{g_i\}_{i=0}^\infty$ are solutions of \eqref{eq_13bis}.  Let $\{f_i\}_{i=0}^\infty$ be the sequence of complex numbers  defined inductively by setting $f_{0}=1$, $f_{1}=t+k$ and $f_{i+2}=  tf_{i+1} - klf_i$, for  $i=0,1,\ldots$. Then, the sequence $\{f_i\}_{i=0}^\infty$ is solution of  \eqref{eq_13bis}. This implies that the sequence $\{\tilde D_i\}_{i=0}^\infty$ with 
\[
\tilde D_i=\begin{pmatrix} h_i & zf_i \\ zf_i & (y+1)f_i\end{pmatrix}, \qquad i=0,1,\dots
\]
is a solution of \eqref{eq_13bis}; here $y=y(t)$ and $z=z(t)$ are defined as in \eqref{eq_8}. 
We claim that $\{D_i\}_{i=0}^\infty$ and $\{\tilde D_i\}_{i=0}^\infty$ coincide. By property i. in Proposition \ref{prop_11} it is enough to check that $D_{0}=\tilde D_{0}$ and $D_{1}=\tilde D_{1}$ to conclude that $D_i=\tilde D_i$, for  $i=0,1,\ldots$.

For $i=0$, we have
\[
D_{0}  =  T-S  =  \begin{pmatrix} x-k & z \\ z & y+1\end{pmatrix}.
\]
Therefore,  $h_{0}=x-k$ and 
\[
D_{0}=\begin{pmatrix} h_{0} & zf_{0} \\ zf_{0} & (y+1)f_{0}\end{pmatrix} = \tilde D_{0}.
\]
 For $i=1$, we have
\[
D_{1} =  TST-STS  =  \begin{pmatrix} kx^2-z^2-k^2x & z(kx-y+k) \\ z(kx-y+k) & kz^2-y^2-y\end{pmatrix}.
\]
Therefore $h_{1}=kx^2-z^2-k^2x$. From  \eqref{eq_43}, we have $kx-y=t$, so $z(kx-y+k)=zf_{1}$. Furthermore, from \eqref{eq_8} we have $z^2=xy+l$ and $x+y=l-1$, from which $kz^2-y^2-y=(y+1)(kx-y+k)=(y+1)f_{1}$. 
Therefore,
\[
D_{1} = \begin{pmatrix} h_{1} & zf_{1} \\ zf_{1} & (y+1)f_{1}\end{pmatrix} = \tilde D_{1},
\]
and we are done.
\end{proof}
%
%
\begin{corollary}\label{lem_6}
Let  $r\geq 1$, $t\in\bC$, $S$ and $T$ as in \eqref{eq_41}, and $D_r=(TS)^rT-(ST)^rS$.  Let $f_r$ be the $r$-th element of the sequence given by Lemma \ref{lem_15}. Then, $SD_{r}=kD_{r}$ if and only if either $f_{r}=0$ or $t=kl+1$.
\end{corollary}
\begin{proof}
Direct computation gives $SD_{r}=\begin{pmatrix} kh_{r} & kzf_{r} \\ -zf_{r} & -(y+1)f_{r}\end{pmatrix}$. Thus, $SD_{r}=kD_{r}$ if and only $-zf_{r}=kzf_{r}$ and $-(y+1)f_{r}=k(y+1)f_{r}$. Since $k\neq -1$, this is possible if and only if either $f_{r}=0$ or $z=y+1=0$. In the latter case, from \eqref{eq_8} we get $0=y+1=\frac{kl+1-t}{k+1}$, that is, $t=kl+1$.  For this value of $t$, again from $\eqref{eq_8}$, we get $z=0$.
\end{proof}
\begin{proposition}
	Let $t=kl+1$ and $S$ and $T$ be as in \eqref{eq_41}. Let $\rho$ be the homomorphism $\rho:\cA_{d}(k,l)\rightarrow\End(\bC^2)$ such that $\rho(M)=S$ and 
	$\rho(N)=T$. Then, the representation $(\bC^2,\rho)$ is  reducible and it is isomorphic to $U_{1,1}\oplus U_{1,2}$. 
	\end{proposition}
\begin{proof}	
From Lemma \ref{lem_23} and Corollary \ref{lem_6}, the matrices $S$ and $T$ satisfy Eqs. \eqref{eq_16}, \eqref{eq_17} and \eqref{eq_18}. Therefore, $(\bC^2,\rho)$ with $\rho(M)=S$ and  $\rho(N)=T$ is a 2-dimensional representation of $\cA_{d}(k,l)$. From \eqref{eq_8}, with  $t=kl+1$, we get $x=l$, $y=-1$ and $z=0$, i.e.,
\[
T=\begin{pmatrix}
		l & 0 \\
		0 & -1
	\end{pmatrix}.
\]
This implies that $\bC^2$ contains one-dimensional submodules, so $(\bC^2,\rho)$ is reducible. 
\end{proof}
In the following we  determine the constraints that the condition $f_{r}=0$ imposes on the complex number $t\in\bC$ so that the matrices $S$ and $T$   satisfy Eq.  \eqref{eq_18}. To this end, we introduce the sequence $\{F_i(X)\}_{i\ge 0}$ of the integer polynomials defined by the three-term recurrence formula
\begin{equation}\label{eq_7}
\begin{array}{c}
F_0(X)=1, \ \ \ \ F_1(X) =  X+k,\\[.2in]
F_{i+2}(X) =  XF_{i+1}(X)-klF_{i}(X), \qquad i=0,1,\ldots.
\end{array}
\end{equation}
\begin{lemma}\label{lem_22}
Let $t\in\bC$ and $\{f_i\}_{i=0}^\infty$ be a solution of the linear recurrence equation \eqref{eq_13}. Then  $f_i=F_i(t)$, for $i=0,1,\ldots$. Therefore,  $f_{r}=0$ implies that $t$ is a zero of $F_r(X)$. 
\end{lemma}

\begin{proof}
 By Lemma \ref{lem_15},   $f_0=1$, $f_1=t+k$, and $f_i=tf_{i-1}-klf_{i-2}$ for $i=2,3,\ldots.$ The result  follows by induction on $i$. It is evident that $f_r=0$ holds whenever $t$ is a zero of $F_r(X)$.
\end{proof}
The following result is  crucial for finding all irreducible two-dimensional representations of $\cA_d(k,l)$.
\begin{proposition}\label{prop_2bis}
	Let $\Gamma=(V,E)$ be an unbalanced distance-biregular graph with diameter $d=2r+2$ and girth deficiency two and let $\cA_{d}(k,l)$ be its adjacency algebra. Let $F_r(X)$ be the $r$-th  polynomial  defined by the  three-term recurrence formula \eqref{eq_7}, $t\in\bC$  a zero of $F_r(X)$ and  $S$ and $T_t$  the $2\times 2$ matrices  as in  \eqref{eq_5}. Then, there exists a 2-dimensional representation $(U_t,\rho_t)$ of $\cA_{d}(k,l)$ such that  $\rho_t(M)=S$ and $\rho_t(N)=T_t$.
\end{proposition}
\begin{proof}
We have to check that the matrices $S$ and $T_t$ satisfy the  equations \eqref{eq_16}, \eqref{eq_17}, and \eqref{eq_18}. By Lemma \ref{lem_23}, 
$S$ and $T_t$ satisfy Eqs. \eqref{eq_16} and \eqref{eq_17}.

Since $f_r=F_r(t)=0$, Corollary \ref{lem_6} gives
$SD_r=kD_r$. Taking transposes and using the symmetry of $S$ and
$D_r$, we also obtain $D_rS=kD_r$. Hence
$(\bC^2,\rho_t)$, with $\rho_t(M)=S$ and $\rho_t(N)=T_t$, is a
2-dimensional representation of $\cA_d(k,l)$.
\end{proof}

 The representation $(U_t,\rho_t)$ given by the previous proposition is called the representation associated with $t$.

In light of Proposition \ref{prop_2bis}, it is now necessary to gather as much information as possible about the zeros of the polynomial $F_r(X)$. We take advantage of the fact that the sequence $\{F_i(X)\}_{i\ge 0}$ is defined by the three-term recurrence formula \eqref{eq_7} with non-constant terms  as well as  the well-known Chebyshev polynomials of the second kind.  We recall that the {\em Chebyshev polynomials of the second kind} are defined by 
\begin{equation}\label{eq_19}
U_n(\cos\theta) \sin \theta= \sin((n+1)\theta)\ \ \mathrm{for \ }n=0,1,\ldots
\end{equation}
These polynomials  satisfy the three-term recurrence formula
\[
U_0(X)=1, \ \ \ \ U_1(X) =  2X,
\] 
\[
U_{n+2}(X) = 2XU_{n+1}(X)-U_{n}(X), \qquad n=0,1,\ldots, 
\]
and the zeros of $U_n(X)$ are:
\[
x_{i}=\cos \left({\frac {i}{n+1}}\pi \right),\qquad i=1,\ldots ,n.
\] 
We refer the reader to \cite{chi} for more information about Chebyshev polynomials.
%
%
\begin{proposition}\label{prop_2}
The zeros $t_1,t_2,\ldots,t_r$ of $F_r(X)$ are real numbers such that
\[
2\sqrt{kl}\cos\left(\frac{r+2-i}{r+1}\pi\right) <t_i<2\sqrt{kl}\cos\left(\frac{r+1-i}{r+1}\pi\right),
\] 
for $i=1,\ldots, r$. Therefore, $0<|t_i|<2\sqrt{kl}$, for $i=1,\ldots,n$.
\end{proposition}
\begin{proof}
 Let $\{U_i(X)\}_{i\ge 0}$  be the sequence of the Chebyshev polynomials of the second kind. Set $\widetilde U_0(X) = 1$, $\widetilde U_1(X) = X$ and 
\begin{equation}\label{eq_37}
	\widetilde U_n(X)=(\sqrt{kl})^nU_n(X/(2\sqrt{kl})), \qquad n=2,\ldots.
\end{equation}
 It is easy to check that $\widetilde U_n(X)$ satisfies the three-term recurrence formula
\[
\widetilde U_{n+1}(X) = X\widetilde U_{n}(X)-kl \widetilde U_{n-1}(X),\qquad  n=1,2,\ldots.
\]  
 Since $F_0(X)=1=\widetilde U_{0}(X)$ and $F_1(X)=X+k=\widetilde U_1(X)+k \widetilde U_{0}(X)$, by induction we have 
\[
F_{n+1}(X)=\widetilde U_{n+1}(X)+k \widetilde U_{n}(X),\qquad n=1,2,\ldots.
\]
In particular,
\begin{equation}\label{eq_20}
F_r(X)=\widetilde U_r(X)+k \widetilde U_{r-1}(X). 
\end{equation} 
For $i=0,\dots,r$ we define
\[
\theta_i = \frac{r+1-i}{r+1}\ \textrm{ and }\ x_i=2\sqrt{kl}\cos\left(\theta_i\pi\right), 
\] 
so that $-2\sqrt{kl}=x_0<x_1<\ldots<x_r$. Let $\sgn$ denote the signum function defined by
\[
\sgn\,x=\left\{\begin{array}{ll}
	1 & \mathrm{if\ }x>0\\
	0 & \mathrm{if\ }x=0\\
	-1 & \mathrm{if\ }x<0\\
\end{array}\right..
\]
We want to compute $\sgn\, F_r(x_i)$ for each $i=0,\dots,r$. 
Since the zeros of $U_r(X)$ are of the form $\cos\left(\frac{i}{r+1}\pi\right)$, we note that $x_1,\ldots,x_r$ are the zeros of $\widetilde U_r(X)$. Therefore, from Eq. \eqref{eq_20} we get $F_r(x_i)=k \widetilde U_{r-1}(x_i)$, for each $i=1,\dots,r$.
So, taking into account Eq. \eqref{eq_37} we have
\[ \sgn\, F_r(x_i)=\sgn\, \widetilde U_{r-1}(x_i) = \sgn\, U_{r-1}(\cos (\theta_i\pi))   \]
Recall that, for any $x>0$,
\[ \sgn\, \sin (x\pi) =\left\{\begin{array}{ll}
	0 & \mathrm{if\ } x\in\bN \\
	(-1)^{\floor{x}} & \mathrm{if\ }x\notin\bN
\end{array}\right.. \]
Fix $i\in\{ 1,\dots,r \}$. By elementary algebraic manipulations we get 
\[ 0 < \frac{r+1-i}{r+1} < 1\ \textrm{ and }\ r-i < \frac{r(r+1-i)}{r+1} < r+1-i, \]
that is $\floor{\theta_i}=0$ and $\floor{r\theta_i}=r-i$. Therefore
\[  
\sgn\, \sin \left(\theta_i\pi\right) = 1 \ \textrm{ and }\ \sgn\, \sin \left(r\theta_i\pi\right) = (-1)^{r-i}.
\]
In conclusion, by Eq. \eqref{eq_19},
\[
\sgn\, U_{r-1}(\cos (\theta_i\pi)) = \sgn\, \sin \left(\theta_i\pi\right) \cdot \sgn\, \sin \left(r\theta_i\pi\right) = (-1)^{r-i}.
\]
Thus, we proved that $\sgn\, F_r(x_i)=(-1)^{r-i}$, for each $i=1,\dots,r$.

To obtain $\sgn\, F_r(x_0)$, we plug into Eq. \eqref{eq_20} the explicit expression of $\widetilde U_r(X)$ obtained from Eqs. \eqref{eq_37} and (\ref{eq_19}) and we get
\[ F_r(2\sqrt{kl}\cos{\theta})\sin\theta = (\sqrt{kl})^r\sin{((r+1)\theta)}+k(\sqrt{kl})^{r-1}\sin{r\theta}. \]
Then,  taking the limit for $\theta\rightarrow\pi$ and using L'H\^{o}pital's Rule, we get 
\begin{align*}
F_r(x_0)=F_r(-2\sqrt{kl})=&\lim_{\theta\rightarrow\pi}{F_r(2\sqrt{kl}\cos{\theta})}\\[.1in]
=&\lim_{\theta\rightarrow\pi}{\left[\frac{(\sqrt{kl})^r}{\sin\theta}\left(\sin{((r+1)\theta)}+\sqrt{k/l}\sin{r\theta}\right)\right]}\\[.1in]
=&\frac{(\sqrt{kl})^r}{\cos\pi}\left((r+1)\cos{((r+1)\pi)}+r\sqrt{k/l}\cos{r\pi}\right)\\[.1in]
=&(-\sqrt{kl})^r(r+1-r\sqrt{k/l}).
\end{align*}

Recall that Corollary \ref{cor_1} gives that $l>k$, so that $r+1-r\sqrt{k/l}>0$. This implies that $\sgn\ F_r(-2\sqrt{kl})=(-1)^r$. Therefore, 
\[
\sgn\, F_r(x_i)=(-1)^{r-i},\ \mathrm{ for }\ i=0,\dots,r.
\]
This yields that $F_r(x_i)$ alternates in sign as $i$ varies from 0 to $r$, that is, $F_r(X)$ has $r$ distinct real zeros, say $t_1,t_2,\ldots,t_r$, with  $t_i\in(x_{i-1},x_i)$, $i=1,\dots,r$. The zeros of $F_r(X)$ separate those of $\widetilde U_r(X)$ in the following way:
	\[
	-2\sqrt{kl}=x_0<t_1<x_1<\cdots <t_r<x_r<2\sqrt{kl}.
	\]
Finally, assume $F_r(0)=0$. From Eq. \eqref{eq_7}, we obtain $0=F_r(0)=(-kl)^{r-1}F_0(0)=(-kl)^{r-1}$,  that contradicts $kl>0$.
\end{proof}
\begin{corollary}\label{cor_5}
	Let $t$ be a root of $F_r(X)$. Then, the complex number $z$ such that $z^2=\frac{(1+kl-t)(t+k+l)}{(k+1)^2}$, is not zero.
\end{corollary}
\begin{proof}
First of all note that $|t|<2\sqrt{kl}$ by Proposition \ref{prop_2}.
If $z=0$, then either $t=kl+1$ or $t=-k-l$. We claim that none of these cases occur. Assume $t=kl+1$. From $kl-2\sqrt{kl} +1=(\sqrt{kl}-1)^2\ge0$ we get 
$|t|=kl +1\ge 2\sqrt{kl}>|t|$. Now assume $t=-k-l$. From $l-2\sqrt{kl}+k=(\sqrt{l}-\sqrt{k})^2\ge0$ we get $|t|=k+l\ge2\sqrt{kl}>|t|$.
\end{proof}
We are now in a position to give the list of all the irreducible representations of $\cA_d(k,l)$.
\begin{theorem}\label{th_8}
	Let $\Gamma=(V,E)$ be an unbalanced distance-biregular graph with diameter $d=2r+2$ and girth deficiency two, and let  $\cA_{d}(k,l)$ be its adjacency algebra.
	Up to isomorphism, $\cA_d(k,l)$ has $r+3$ finite dimensional irreducible representations. These are the three 1-dimensional representations $U_{1,1}$, $U_{1,2}$ and $U_{1,3}$ described in Proposition \ref{prop_3} and the 2-dimensional  representations $U_{t_1},\dots,U_{t_r}$ corresponding to the roots of the $r$-th polynomial $F_r(X)$ defined by the  three-term recurrence formula \eqref{eq_7}.
\end{theorem}
\begin{proof}
	The three 1-dimensional representations described in Proposition \ref{prop_3} are clearly irreducible. Let $t_1,\ldots, t_r$ be the roots of $F_r(X)$. By Proposition \ref{prop_2}, these are pairwise distinct real numbers, and,  by Proposition \ref{prop_2bis}, for each $t_i$ there is a 2-dimensional representation $U_{t_i}$ of $\cA_d(k,l)$. 
	
	Assume that $U_{t_i}$ and  $U_{t_j}$ are isomorphic for distinct roots $t_i$ and $t_j$  of $F_r(X)$. Set $\rho_i(N)=T_{t_i}$ and $\rho_j(N)=T_{t_j}$, where $T_{t_i}$ and $T_{t_j}$ are as in \eqref{eq_5}.
	Since the trace of a square matrix is invariant under similarity transformations, we would have $t_i=\Trace(ST_{t_i})=\Trace(ST_{t_j})=t_j$, a contradiction. Therefore, the representations $U_{t_i}$, $i=1,\ldots,r$, are pairwise non-isomorphic.
	
	We now prove that each $U_{t_i}$ is irreducible. Since  $\dim\, U_{t_i}=2$, then any non-trivial subrepresentation of $U_{t_i}$ should have dimension 1, i.e.,  $S$ and $T_{t_i}$  should have a common eigenvector. Since $S$ is a diagonal matrix and $z(t_i)\neq 0$ by Corollary \ref{cor_5}, it is easy to check that $S$ and $T_{t_i}$ have no eigenvector in common.

	Summarizing, there are  three distinct 1-dimensional (irreducible) representations and $r$ non-isomorphic 2-dimensional irreducible representations of $\cA_d(k,l)$. Note that 
	\[
	(\dim\, U_{1,1})^2+(\dim\, U_{1,2})^2+(\dim\, U_{1,3})^2+\sum_{i=1}^r{(\dim\, U_{t_i})^2}=3+4r= 2d-1=\dim\,\cA_d(k,l).
	\]
	Finally, by Eq. \eqref{eq_28} we see that $\cA_d(k,l)$ has no other irreducible representations, up to isomorphism.
\end{proof}
\section{The multiplicities of the irreducible representations  of $\cA_d(k,l)$}\label{sec_5}
 {Let  $\cA_{d}(k,l)$ be the adjacency algebra of an unbalanced distance-biregular graph  $\Gamma=(V,E)$   with diameter $d=2r+2$ and girth deficiency two.
In this section we evaluate the multiplicities of the decomposition of the standard module $\bC^E$ of $\cA_{d}(k,l)$ into irreducible modules.}

For any irreducible representation $(U,\rho)$ of $\cA_d(k,l)$, let
$m$ denote its multiplicity as a submodule of $\bC^E$, and let
$\zeta:\cA_d(k,l)\to\bC$, $A\mapsto\Trace(\rho(A))$, be its character. In \cite{hig3} it was shown that $m$ and $\zeta$ are related by the following {\em orthogonality relation}:
\begin{equation}\label{eq_10}
\sum_{s\in \cR}{\frac{\overline{\zeta(A_s)}\zeta(A_s)}{\eta_s}}=\frac{\zeta(I)}{m}|\Omega|,
\end{equation}
where $\eta_s$ is the valency of the relation $s$. Note that $\zeta(I)=\dim\, U$.

By taking into account Lemma \ref{lem_5},  {Corollary \ref{rem_3a}, and that the trace of a square matrix is invariant under transposition}, we get
\[
\zeta((MN)^i)=\Trace(\rho(MN)^i)
=\Trace(\rho(NM)^i)
=\zeta((NM)^i),
\]
for $i=1,\ldots,r$. Since  $\eta_sv=A_sv$, where $v=\sum_{x\in E}{\chi_x}$, formula \eqref{eq_10} becomes
\begin{equation}\label{eq_12}
\begin{split}
\frac{\zeta(I)}{m}|E|=&\zeta(I)^2+2\sum_{i=1}^{r}{\frac{\zeta((MN)^i)\overline{\zeta((MN)^i)}}{k^il^i}}+\frac{1}{k}\sum_{i=0}^{r}{\frac{\zeta((MN)^iM)\overline{\zeta((MN)^iM)}}{k^{i}l^i}}\\& 
+\frac{1}{l}\sum_{i=0}^{r-1}{\frac{\zeta((NM)^iN)\overline{\zeta((NM)^iN})}{k^{i}l^{i}}}+\frac{\zeta(D)\overline{\zeta(D)}}{k^rl^r(l-k)}.
\end{split}
\end{equation}
In the following we will use the  formula \eqref{eq_12} to calculate the multiplicities  of the irreducible representations of the adjacency algebra $\cA_d(k,l)$. We start with the 1-dimensional representations of $\cA_d(k,l)$, which are given in Proposition \ref{prop_3}.
\begin{proposition}
Let $m_{1,i}$ be the multiplicity of the 1-dimensional representation $U_{1,i}$ of $\cA_d(k,l)$  {given by Proposition \ref{prop_3}}, for $i=1,2,3$. Then 
\begin{align*}
m_{1,1}=& 1 ,\\
m_{1,2}= & \frac{k^{r+1}l^r(l+1)}{k+1},\\
m_{1,3}= & \frac{(k^{r+1}l^r(l+1)-k-1)(l-k)}{(kl-1)(k+1)}.
\end{align*}
\end{proposition}
\begin{proof}
Clearly,  $m_{1,1}=1$ since $U_{1,1}$ is the principal representation. For the 1-dimensional representations $U_{1,2}$ and $U_{1,3}$ of $\cA_d(k,l)$, the values taken by the corresponding characters $\zeta_{1,2}$ and $\zeta_{1,3}$ are given in Table \ref{tab_2}. By plugging them in \eqref{eq_12}, the result follows by standard algebraic manipulations of geometric sums.
\end{proof}
\begin{remark}{\em
Note that $m_{1,2}$ is an integer because of the feasibility condition in Corollary \ref{cor_2}. 
} 
\end{remark}
\begin{proposition}
 {The value} $m_{1,3}$ is an integer.
\end{proposition}
\begin{proof}
We write
\[
\begin{split}
\frac{k^{r+1}l^r(l+1)-k-1}{(kl-1)}
              &= \frac{k^{r+1}l^{r+1}-1}{kl-1}+k\frac{k^rl^{r}-1}{kl-1}\\
              &= \sum_{i=0}^{r}{k^il^i}+k\sum_{i=0}^{r-1}{k^il^i}\\
              &= k^rl^r+(k+1)\sum_{i=0}^{r-1}{k^il^i}.
\end{split}
\]
Therefore, 
\[
\frac{(k^{r+1}l^r(l+1)-k-1)(l-k)}{kl-1} \equiv (-l)^r(l+1)\, \mod\,(k+1)
\]
which is congruent to 0 mod $k+1$ by Corollary \ref{cor_2}.
\end{proof}
 {We now consider the 2-dimensional irreducible representations of $\cA_d(k,l)$ which, by Theorem \ref{th_8}, are in 1-1 correspondence with the zeros of the polynomial $F_r(X)$. Recall that the   zeros of $F_r(X)$ are $r$ distinct real numbers by Proposition \ref{prop_2}.} 
 {Let $t$ be a zero of $F_r(X)$ and $(U_t,\rho_t)$ be the irreducible representation of $\cA_d(k,l)$ associated with it. Then  $\rho_t(M)=S$ and $\rho_t(M)=T_t$, where $S$ and $T_t$ are the matrices defined by \eqref{eq_5}. In order to compute its multiplicity, in light of Proposition \ref{prop_10}, we again make use  of the theory of linear recursion \cite{kp}.} 
\begin{remark}\label{rem_6}
{\em For any zero $t$ of $F_r(X)$, the discriminant of the polynomial $X^2-tX+kl$ is $\Delta_t=t^2-4kl$. By Proposition \ref{prop_2}, we have $|t|<2\sqrt{kl}$ so that $\Delta_t=t^2-4kl$ is negative. Therefore, $X^2-tX+kl$ has two distinct complex roots: $\lambda\in\bC\setminus\bR$ and its complex conjugate $\bar\lambda=kl/\lambda$.}
\end{remark}
\begin{lemma}\label{lem_7}
 Let $t$ be a zero of $F_r(X)$, and $\lambda$ a zero of $X^2-tX+kl$.  Then, for any  complex solution $\{a_i\}_{i=0}^\infty$ of the linear recurrence equation \eqref{eq_13} we have
\begin{equation}\label{eq_22}
a_i=\frac{1}{\lambda-\bar\lambda}\left[\lambda^i(a_1-\bar\lambda a_0)-\bar\lambda^i(a_1-\lambda a_0)\right]. 
\end{equation}
\end{lemma}
\begin{proof}
 {The characteristic polynomial of the linear recurrence equation \eqref{eq_13} is $X^2-tX+kl$. By Remark \ref{rem_6}, this polynomial has two complex  {distinct} roots $\lambda$ and $\bar\lambda$.  {By property ii. in Proposition \ref{prop_11}, }every solution $\{a_i\}_{i=0}^\infty$ of \eqref{eq_13} can be written as $a_i= \beta \lambda^i + \gamma \bar\lambda^i$ for any $i\geq0$ and for some $\beta,\gamma\in\bC$. In particular, for $i=0$ and $i=1$ we get 
\[
\left\{
\begin{array}{lcl}
	a_0 & = & \beta + \gamma \\[.08in]
	a_1 & = & \beta \lambda + \gamma \bar\lambda
\end{array}\right..
\]
Solving for $\beta$ and $\gamma$,  we get
\[ \beta = \frac{a_1-\bar\lambda a_0}{\lambda-\bar\lambda}, \qquad\gamma= \frac{\lambda a_0-a_1}{\lambda-\bar\lambda},  \]
 {whence equality \eqref{eq_22}.}}
\end{proof}
\begin{theorem}\label{th_4}
Let $t$ be a zero of $F_r(X)$, let $(U,\rho)$ be the irreducible
2-dimensional representation of $\cA_d(k,l)$ associated with it, and
let $\zeta$ be its character. Let $\lambda$ denote a zero of
$X^2-tX+kl$. Then 
\begin{align*}
\zeta((MN)^i)&=\zeta((NM)^i)=\lambda^i+\bar\lambda^i,\\
\zeta((MN)^iM)&=\frac{1}{\lambda-\bar\lambda}\left[\lambda^i(\lambda(k-1)+k(l-1))-\bar\lambda^i(\bar\lambda(k-1)+k(l-1))\right],\\
\zeta((NM)^iN)&=\frac{1}{\lambda-\bar\lambda}\left[\lambda^i(\lambda(l-1)+l(k-1))-\bar\lambda^i(\bar\lambda(l-1)+l(k-1))\right],\\
\zeta(D)&=\frac{l-k}{\lambda-\bar\lambda}\left[\lambda^r(\lambda-1)-\bar\lambda^r(\bar\lambda-1)\right].
\end{align*}
\end{theorem}
\begin{proof}
From Proposition \ref{prop_2bis} we have $\rho(M)=S$ and $\rho(N)=T$, where $S$ and $T$ are as in \eqref{eq_41}.
Recall also that we set $D_i=(TS)^iT-(ST)^iS$, $i= 0,1,\ldots$.
By direct computation, we have
\[
\Trace(S)=k-1, \ \  \Trace(T)=l-1, \ \ \Trace(ST)=\Trace(TS)=t,
\]
\[
\Trace(STS)=t(k-1)+k(l-1), \ \  \Trace(TST)=t(l-1)+l(k-1),
\]
\[
 {\Trace(D_0) = l-k, \ \ \Trace(D_1) = (t-1)(l-k) .  \ \ }
\]
By Proposition \ref{prop_10}, the sequences 
 \[
\{\Trace((ST)^i)\}_{i=0}^\infty,\ \{\Trace((ST)^iS)\}_{i=0}^\infty,\ \{\Trace((TS)^iT)\}_{i=0}^\infty \textrm{ and } \{\Trace(D_i)\}_{i=0}^\infty
\] 
 satisfy the recurrence formula \eqref{eq_13}, and, hence, the closed form expression \eqref{eq_22} 
\[
a_i=\frac{1}{\lambda-\bar\lambda}\left[\lambda^i(a_1-\bar\lambda a_0)-\bar\lambda^i(a_1-\lambda a_0)\right]
\]
 by Lemma \ref{lem_7}. 

Therefore, taking into account that $t=\lambda+\bar\lambda$ and $\lambda\bar\lambda=kl$, we get 
\[
\begin{split}
\zeta((MN)^i)=\Trace((ST)^i) 
& {=\frac{1}{\lambda-\bar\lambda}\left[\lambda^i(t-2\bar\lambda )-\bar\lambda^i(t-2\lambda )\right]}\\
&=\frac{1}{\lambda-\bar\lambda}[\lambda^i(\lambda-\bar\lambda)+\bar\lambda^i(\lambda-\bar\lambda)]\\
&=\lambda^i+\bar\lambda^i.
\end{split}
\]

 {The other cases can be obtained with a similar computation. For the last case, recall that $\rho(D)=D_{r}$.}
\end{proof}
\begin{lemma}\label{lem_8}
Let $t_1,\ldots, t_r$ be the zeros of $F_r(X)$, and set
\begin{equation}\label{eq_33}
p(X)=(X^2-kl)\prod_{j=1}^{r}{(X^2-t_jX+kl)}.
\end{equation}
Then, the zeros of $p(X)$ are all distinct and we have
\begin{equation}\label{eq_26}
p(X)=X^{2r+2}+kX^{2r+1}-k^{r+1}l^rX-k^{r+1}l^{r+1}.
\end{equation} 
\end{lemma}
\begin{proof}
 {For any zero $t_j$ of $F_r(X)$, let $\lambda_j$ and $\overline\lambda_j$ be the zeros of $X^2-t_jX+kl$, $j=1,\dots,r$.} By Remark \ref{rem_6}, $\lambda_j,\bar\lambda_j\neq\pm\sqrt{kl}$ and $\lambda_j\neq\bar\lambda_j$. Assume $\lambda_i=\lambda_j$ or $\lambda_i=\bar\lambda_j$ for some $i\neq j$. This implies $t_i=\lambda_i+\bar\lambda_i=\lambda_j+\bar\lambda_j=t_j$, which is not possible by Proposition \ref{prop_2}. Therefore, all zeroes of $p(X)$ are distinct.

By the definition \eqref{eq_7} of the sequence of polynomials $\{F_i(X)\}_{i=0}^\infty$, the sequence $\{F_i(t_j)\}_{i=0}^\infty$ satisfies the recurrence formula \eqref{eq_13}, whence it satisfies the closed form expression \eqref{eq_22} given by Lemma \ref{lem_7}. Therefore, taking into account that $F_0(t_j)=1$, $F_1(t_j)=t_j+k$ and $t_j=\lambda_j+\bar\lambda_j$, we have
\[
\begin{split}
0=(\lambda_j-\bar\lambda_j)F_r(t_j)
&=\lambda_j^r
  \bigl(F_1(t_j)-\bar\lambda_jF_0(t_j)\bigr)
 -\bar\lambda_j^r
  \bigl(F_1(t_j)-\lambda_jF_0(t_j)\bigr)\\
&=\lambda_j^r(\lambda_j+k)
 -\bar\lambda_j^r(\bar\lambda_j+k).
\end{split}
\]
 {By multiplying both sides of the latter equation by $\lambda_j^{r+1}$, and taking into account that $\lambda_j\bar\lambda_j=kl$, we see that $\lambda_j$ is a zero of the polynomial 
\[
\widetilde p(X)=X^{2r+2}+kX^{2r+1}-k^{r+1}l^rX-k^{r+1}l^{r+1}.
\] 
}
A very similar argument shows that  $\bar\lambda_j$ is also a zero of $\widetilde p(X)$.

It is easy to check that $\pm\sqrt{kl}$ are also zeros of $\widetilde p(X)$. Therefore, $p(X)$ and $\widetilde p(X)$  are monic polynomials with the same zeros, therefore they coincide.
\end{proof}

\begin{theorem}\label{th_5}
Let $t$ be a zero of $F_r(X)$ and  $U_t$ the irreducible  2-dimensional  representation of $\cA_d(k,l)$ associated with $t$. Then, the multiplicity $m_t$  of $U_t$ as a submodule of $\bC^E$ is
\begin{equation}\label{eq_21}
m_t=|E|\frac{(t^2-4kl)}{(t-kl-1)((2r+1)(t+k+l)+l-k)}.
\end{equation}
\end{theorem}
\begin{proof}
Let $\zeta$ be the character of $U_t$. 
Note that all values of $\zeta$ listed in Theorem \ref{th_4} are fixed
by complex conjugation and hence are real numbers. Therefore,
Eq. \eqref{eq_12} becomes
\begin{equation}\label{eq_12bis}
\begin{split}
\frac{\zeta(I)}{m_t}|E|=&\zeta(I)^2+2\sum_{i=1}^{r}{\frac{\zeta((MN)^i)^2}{k^il^i}}+\frac{1}{k}\sum_{i=0}^{r}{\frac{\zeta((MN)^iM)^2}{k^{i}l^i}}\\& 
+\frac{1}{l}\sum_{i=0}^{r-1}{\frac{\zeta((NM)^iN)^2}{k^{i}l^{i}}}+\frac{1}{k^rl^r(l-k)}\zeta(D)^2.
\end{split}
\end{equation}

 {Let $\lambda$ and $\bar\lambda$ be the zeros of $X^2-tX+kl$, that is $\lambda+\bar\lambda=t$ and $\lambda\bar\lambda=kl$.
The main tool we use to evaluate Eq. \eqref{eq_12bis} is the following observation:} by Lemma \ref{lem_8}, $\lambda$ is a zero of the polynomial $p(X)$, so we have 
\[
0=p(\lambda)=\lambda^{2r+1}(\lambda+k)-k^{r+1}l^r(\lambda+l),
\]
which  {leads} to
\begin{equation*}
	 { \frac{\lambda^{2r+1}}{k^rl^r}=\frac{k(\lambda+l)}{\lambda+k} . }
\end{equation*}
 {Therefore, 
\begin{equation}\label{eq_38}
	\sum_{i=0}^{r}\frac{\lambda^{2i}}{k^il^i} 
	= \frac{\left(\frac{\lambda^{2}}{kl}\right)^{r+1}-1}{\frac{\lambda^{2}}{kl}-1}
	= \frac{\lambda\left(\frac{\lambda^{2r+1}}{k^rl^r}\right)-kl}{\lambda^2-kl} 
	= \frac{\frac{\lambda k(\lambda+l)}{\lambda+k}-kl}{\lambda^2-kl} 
	= \frac{k}{\lambda+k}
\end{equation}
We will take into account Theorem \ref{th_4}} and make repeated use of Eq. \eqref{eq_38} and of its conjugate in the following computations.
 
The first summation of Eq. \eqref{eq_12bis} becomes
\[
\sum_{i=0}^{r}{\frac{\zeta((MN)^i)^2}{k^il^i}} =\sum_{i=0}^{r}{\left(\frac{\lambda^{2i}}{k^il^i}+\frac{\bar\lambda^{2i}}{k^il^i}+2\right)}
=\frac{k}{\lambda+k}+\frac{k}{\bar\lambda+k}+2(r+1).
\]
To compute the other sums in Eq. \eqref{eq_12bis}, we use $(\lambda-\bar\lambda)^2  {= (\lambda + \bar\lambda)^2 - 4\lambda\bar\lambda} =t^2-4kl$.  Moreover, to keep the notation more compact, we set
\[
\begin{array}{c}
	a_{k,l}=k-1, \qquad b_{k,l}=k(l-1),\\[.2in]
	c_{k,l}=(\lambda a_{k,l}+b_{k,l})(\bar\lambda a_{k,l}+b_{k,l})=kl(k-1)^2+tk(k-1)(l-1)+k^2(l-1)^2.
\end{array}
\]
Hence we have:
\[
\begin{aligned}
\sum_{i=0}^{r}\frac{\zeta((MN)^iM)^2}{k^il^i}
&=\frac{1}{(\lambda-\bar\lambda)^2}
  \sum_{i=0}^{r}
  \frac{\left[
  \lambda^i(a_{k,l}\lambda+b_{k,l})
  -\bar\lambda^i(a_{k,l}\bar\lambda+b_{k,l})
  \right]^2}{k^il^i}\\
&=\frac{1}{t^2-4kl}\Bigg[
  (a_{k,l}\lambda+b_{k,l})^2
  \sum_{i=0}^{r}\frac{\lambda^{2i}}{k^il^i}\\
&\qquad+
  (a_{k,l}\bar\lambda+b_{k,l})^2
  \sum_{i=0}^{r}\frac{\bar\lambda^{2i}}{k^il^i}
  -2\sum_{i=0}^{r}c_{k,l}\Bigg]\\
&=\frac{1}{t^2-4kl}\Bigg[
  (a_{k,l}\lambda+b_{k,l})^2\frac{k}{\lambda+k}\\
&\qquad+
  (a_{k,l}\bar\lambda+b_{k,l})^2\frac{k}{\bar\lambda+k}
  -2(r+1)c_{k,l}\Bigg].
\end{aligned}
\]

Similarly,
\[
\begin{aligned}
\sum_{i=0}^{r-1}\frac{\zeta((NM)^iN)^2}{k^il^i}
=-\frac{1}{t^2-4kl}\Bigg[
&(a'_{k,l}\lambda+b'_{k,l})^2
  \frac{kl}{\lambda(\lambda+k)}\\
&+(a'_{k,l}\bar\lambda+b'_{k,l})^2
  \frac{kl}{\bar\lambda(\bar\lambda+k)}
  +2rc'_{k,l}\Bigg].
\end{aligned}
\] 
where 
\[
\begin{array}{c}
a'_{k,l}=l-1, \ \ \ \ b'_{k,l}=l(k-1),\\[.2in]
c'_{k,l}=(\lambda a'_{k,l}+b'_{k,l})(\bar\lambda a'_{k,l}+b'_{k,l})=kl(l-1)^2+tl(k-1)(l-1)+l^2(k-1)^2, 
\end{array}
\]
and
\[
\frac{\zeta(D)^2}{k^rl^r(l-k)}=\frac{l-k}{t^2-4kl}\left[\frac{k(\lambda+l)}{\lambda(\lambda+k)}(\lambda-1)^2+\frac{k(\bar\lambda+l)}{\bar\lambda(\bar\lambda+k)}(\bar\lambda-1)^2-2(\lambda-1)(\bar\lambda-1)\right].
\]
Finally, by plugging  all the above values in \eqref{eq_12bis}, with the assistance of the computer algebra system MAGMA \cite{magma}, we get the result.
\end{proof}
 {The following easy corollary of the previous theorem is the main tool we use in the following section.}
\begin{corollary}\label{cor_4}
Every irreducible factor of $F_r(X)$ has degree at most two over the integers.
\end{corollary}
\begin{proof}
Clearing the denominator in expression \eqref{eq_21}, we see that
every zero $t$ of $F_r(X)$ is a zero of a polynomial of degree at most
two over the integers.
\end{proof}
\section {The diameter of an unbalanced\\
distance-biregular graph with girth deficiency two}
\label{sec_6}

In this section we provide the feasible values for the diameter  {$d=2r+2$} of an unbalanced distance-biregular graph with girth deficiency two.  To facilitate the achievement of our goal, we introduce  the polynomial
\[
q(X)=lX^{2r+2}-kX^{2r+1}-2(l-k)X^{r+1}-kX+l.
\]
We will show that the zeros of $q(X)$ are related to the zeros of the polynomial $p(X)$ defined in \eqref{eq_33}, and, hence, depend on the zeros of $F_r(X)$ by Lemma \ref{lem_8}.
\begin{proposition}\label{prop_5}
Let $t_i$ be a zero of $F_r(X)$, and $\lambda_i$ a zero of $X^2-t_iX+kl$. Set $\widetilde\lambda_i=\lambda_i^2/(kl)$. Then, the zeros of $q(X)$ are $1$ (with multiplicity 2), and the $2r$ distinct numbers $\widetilde \lambda_i$ and  $\widetilde\lambda_i^{-1}$, for $i=1,\ldots,r$.
\end{proposition}
\begin{proof}
Set $Y=X/\sqrt{kl}$. By applying this change of variable in \eqref{eq_26},  {for any zero $\lambda_i$ of $p(X)$ we obtain} that $\lambda'_i=\lambda_i/\sqrt{kl}$ is a zero of
\[
  {\frac{p(\sqrt{kl}Y)}{k^{r+1}l^r}=} lY^{2r+2}+\sqrt{kl}Y^{2r+1}-\sqrt{kl}Y-l,
\]
that is,
\[
l({\lambda'_i}^{2r+2}-1)=-\sqrt{kl}{\lambda'_i}({\lambda'_i}^{2r}-1).
\]
Squaring both sides, we obtain
\[
l({\lambda'_i}^{4r+4}-2{\lambda'_i}^{2r+2}+1)=k{\lambda'_i}^2({\lambda'_i}^{4r}-2{\lambda'_i}^{2r}+1),
\] 
giving that $\lambda'_i$ is a zero of
\[
 {q(X^2)=}lX^{4r+4}-kX^{4r+2}-2(l-k)X^{2r+2}-kX^2+l.
\]

From here it is easily seen that $\widetilde\lambda_i$ is a zero of
$q(X)$, since $\widetilde\lambda_i={\lambda'_i}^2$.

 {By direct computation, 1 is a zero of both $q(X)$ and its first derivative. So, to finish the proof it is enough to show that $\widetilde \lambda_1, \widetilde\lambda_1^{-1},\dots, \widetilde \lambda_r, \widetilde\lambda_r^{-1}$ are distinct numbers.}

Let $t_i$ be a zero of $F_r(X)$, and  $\lambda_i$ and $\bar\lambda_i$ the zeros of $X^2-t_iX+kl$. Note that $\overline{\widetilde\lambda_i}=\widetilde\lambda_i^{-1}$ since $\lambda_i\bar\lambda_i=kl$.
By previous arguments, $\widetilde\lambda_i$ and $\widetilde\lambda_i^{-1}$ are zeros of $q(X)$. Assume $\widetilde\lambda_i=\widetilde\lambda_i^{-1}$, i.e., $\widetilde\lambda_i\in\{\pm 1\}$, thus $\lambda_i^2\in\{\pm kl\}$  as $\widetilde\lambda_i=\lambda_i^2/(kl)$. Then, $\lambda_i^2=-kl$ leads to $t_i=0$, which is not possible by Proposition \ref{prop_2}, and $\lambda_i^2=kl$ leads to $t_i^2=4kl$, which is not possible by Remark \ref{rem_6}.

 Assume that $\widetilde\lambda_j=\widetilde\lambda_i$ or $\widetilde\lambda_j = \widetilde\lambda_i^{-1}=\overline{\widetilde\lambda_i}$, for some $i\neq j$. Since the roots of $p(X)$ are all distinct, this means that either $\lambda_j=-\lambda_i$ or $\lambda_j=-\bar\lambda_i$. In any case, this gives that both $\lambda_i$ and $-\lambda_i$ are zeros of $p(X)$. The simultaneous solution of $p(\lambda_i)+p(-\lambda_i)=0$ and $p(\lambda_i)-p(-\lambda_i)=0$ leads to $\lambda_i^2=kl$, which we already excluded.
This finishes the proof.
\end{proof}
 Since we are considering unbalanced distance-biregular graphs with
diameter $d=2r+2$ and girth deficiency two, and since the degree of
$q(X)$ is precisely $d$, it is essential to study the factorization
of this polynomial over the integers to obtain as much information as
possible about $d$. This is the goal of the next lemmas.
By $\bZ$, $\bQ$ and $\bF_{s}$, $s=p^h$ a prime power, we denote the ring of integers, the field of the rationals and the field with $s$ elements, respectively. 

Given a polynomial $f(X)=a_nX^n+\cdots+a_0$, its {\em reciprocal polynomial}  is the polynomial
\[
f^*(X)= X^{n}f(1/X)=a_0X^n+a_{1}X^{n-1}+\ldots+ a_{n-1}X+a_n.
\]
A polynomial $f(X)$ is said to be {\em palindromic} (or {\em self-reciprocal}) if $f(X)= f^*(X)$. It is evident that the coefficients of a palindromic polynomial satisfy $a_i=a_{n-i}$, for all $i$. 

In the next two lemmas we use some basic facts about field extensions;
see \cite[Chapter 15]{art}.
\begin{lemma}\label{prop_6}
Let $\Gamma$ be an unbalanced distance-biregular graph with diameter $d=2r+2$, girth deficiency two and vertex degrees $k+1$ and $l+1$. Let 
\[q(X)=m(X-1)^2q_1(X)\cdots q_n(X),\]
$m\in \bZ$, be the irreducible factorization of $q(X)$ over $\bZ$. Then, $\deg\, q_i(X)\in\{2,4\}$ and $q_i(X)$ is palindromic for each $i=1,\dots,n$.
\end{lemma}
\begin{proof}
Since the greatest common divisor of the coefficients of $q_i(X)$ is
one, $q_i(X)$ is primitive for each $i=1,\ldots,n$. Hence, by Gauss's Lemma, it is irreducible over $\bQ$.
Let $t$ be any zero of $F_r(X)$, and $\lambda$  {be} a zero of $X^2-tX+kl$. Then, $\widetilde\lambda=\lambda^2/(kl)$ is a zero of $q(X)$ by Proposition \ref{prop_5}.

By Corollary \ref{cor_4}, $t$ is a root of a degree-$2$ polynomial with integer coefficients, and hence of a polynomial over $\bQ$.
 This gives $[\bQ(t):\bQ]\in \{1,2\}$, and therefore,  $[\bQ(\lambda):\bQ]\in \{1,2,4\}$.
Since $\widetilde\lambda\in \bQ(\lambda)$, then $[\bQ(\widetilde\lambda):\bQ]$ divides $[\bQ(\lambda):\bQ]$.
This implies that the minimal polynomial of $\widetilde\lambda$ over $\bQ$ has degree in $\{1,2,4\}$ and divides $q(X)$.
By Remark \ref{rem_6} the real part of $\lambda$ is non-zero, so $\widetilde \lambda\in\bC\setminus\bR$. This yields  $\deg\, q_i(X)\in\{2,4\}$.

Clearly $q(X)$ is palindromic. Let $\widetilde\lambda$ be a zero of  {$q_i(X)=a_0X^{n_i}+\dots+a_{n_i}$, $n_i=\deg\,q_i(X)$. Since $\widetilde{\lambda}\notin \bR$}, then $\widetilde\lambda^{-1}=\bar{\widetilde\lambda}\neq \widetilde\lambda$ is also a zero of $q_i(X)$. By Vi\`ete's Formula,  {$a_{n_i}/a_0$ is equal to the product of all the zeros of $q_i(X)$, and this} yields $a_0=a_{n_i}$.

Let $q_i^*(X)$ be the reciprocal polynomial of $q_i(X)$. We have 
\[
q_i^*(\widetilde \lambda)={\widetilde \lambda}^{n_i}q_i({\widetilde \lambda}^{-1})={\widetilde \lambda}^{n_i}q_i(\,\overline{\widetilde \lambda}\,)={\widetilde \lambda}^{n_i}\overline{q_i(\widetilde \lambda)}=0.
\]
Thus $q_i^*(X)$ and $q_i(X)$ have a common zero. Since $q_i(X)$ is
irreducible over $\bQ$ and the two polynomials have the same degree,
they are scalar multiples of one another. Their leading coefficients
are equal, so $q_i^*(X)=q_i(X)$.
\end{proof}
 In general, determining the degrees of the irreducible factors of a polynomial over the integers is a difficult problem. Nevertheless, the same problem is easier to tackle over finite fields. Thus, in the following we will determine the properties of the reduction of $q(X)$ modulo $p$, for various prime numbers $p$.
For any integer $n\in\bZ$ we denote by $\widehat n\in\bZ_p$ its reduction modulo $p$. Similarly, for any polynomial $P(X)\in \bZ[X]$, its reduction modulo $p$ will be denoted by $\widehat P(X)\in \bZ_p[X]$.
\begin{lemma}\label{lem_14}
Let $\Gamma$ be an unbalanced distance-biregular graph with diameter $d=2r+2$, girth deficiency two  and vertex degrees $k+1$ and $l+1$, and let $q(X)=m(X-1)^2q_1(X)\cdots q_n(X)$, $m\in\bZ$, be the irreducible factorization of $q(X)$  over $\bZ$. Then, for any prime $p$ and any $i=1,\ldots,n$, the following hold:
\begin{itemize}
\item[1.] the splitting field of $\widehat q_i(X)\in\bZ_p[X]$ is a subfield of $\bF_{p^4}$, 
\item[2.] any nonzero root of $\widehat q_i(X)$ is a zero of $(X^{p^2-1}-1)(X^{p^2+1}-1)$.
\end{itemize}
\end{lemma}
\begin{proof}
By Lemma \ref{prop_6}, $q_i(X)$ is palindromic with $\deg\,q_i(X)\in\{2,4\}$.   Clearly, $\widehat q_i(X)$ is also palindromic with $\deg\, \widehat q_i(X)\leq \deg\,q_i(X)$. This implies that the zeros of $\widehat q_i(X)$  are either 1 or $-1$  or they come in pairs $(\alpha,1/\alpha)$. Consequently,  the non-constant polynomials of the irreducible factorization of $\widehat q_i(X)$ over $\Fp$ necessarily have degree 1, 2 or 4.

Let $z\in \bF_{p^4}$ be a nonzero zero of  {$\widehat q_i(X)$}. If the splitting field of $\widehat q_i(X)$ is a subfield of $\bF_{p^2}$, then $z^{p^2-1}=1$. Otherwise, $\widehat q_i(X)$ is an irreducible polynomial of degree 4 over $\Fp$, and   its zeros are $z,z^p,z^{p^2}, z^{p^3}$.  Since  $\widehat q_i(X)$ is palindromic then $z^{-1}$ is also a zero of $\widehat q_i(X)$, and it is easy to check that the only possibility is  $z^{-1}=z^{p^2}$, i.e. $z^{p^2+1}=1$.
\end{proof}
We are now in position to give  constraints on the diameter of an unbalanced distance-biregular graph of girth deficiency two: this is done  in  Lemma \ref{lem_12}, Corollary \ref{cor_6} and Proposition \ref{prop_9}.  

Let $p$ be a prime number. Then, any non-zero integer $n\in\bZ$ can be factored as $n=p^as$ for some $a\geq 0$ and $p\nmid s$. With this notation, the $p${\em -adic valuation of} $n$ is $\nu_p(n)=a$. Note that $s=n/p^{\nu_p(n)}$.
For any  {prime $p$,} we define the set
\[
U_p=\left\{n\in\bZ\setminus\{0\}: \frac{n}{p^{\nu_p(n)}} \mathrm{\ divides\ either\ }p^2-1 \mathrm{\ or\ }p^2+1 \right\}.
\]
Let  {$m=\GCD(k,l)$, $L=l/m$, $K=k/m$ and set}
\begin{equation}\label{eq_44}
	Q(X)=q(X)/m=LX^{2r+2}-KX^{2r+1}-2(L-K)X^{r+1}-KX+ L.
\end{equation}
The polynomials $Q(X)$ and $q(X)$ have the same zeros, and $\GCD(K,L)=1$, so $K$, $L$ and $L-K$ are pairwise coprime.

\begin{lemma}\label{lem_12}
Let $\Gamma$ be an unbalanced distance-biregular graph with diameter $d=2r+2$, girth deficiency two  and vertex degrees $\beta_0^A=mK+1$ and $\beta_0^B=mL+1$, with $\GCD(K,L)=1$. Then, for any prime   $p$ the following hold: 
\begin{enumerate}
\item\label{con_1} If $p|L$, then $r\in U_p$;
\item\label{con_2} If $p|K$, then $r+1\in U_p$;
\item\label{con_3} If $p|L-K$, then $2r+1\in U_p$.
\end{enumerate}
\end{lemma}
\begin{proof}

If $p|L$,  { i.e. $\widehat L=0$, then the reduction modulo $p$ of $Q(X)$ is
\[
 \widehat Q(X)= -\widehat K X^{2r+1}+ 2\widehat KX^{r+1}-\widehat KX = -\widehat KX(X^{r}-1)^2 = -\widehat KX(X^{\frac{r}{p^{\nu}}}-1)^{2p^{\nu}}\in \Fp[X],  
\]
where $\nu = \nu_p(r)$.}
Let $\eta$ be a primitive $r/p^{\nu}$-th root of unity in the algebraic closure of $\Fp$. Obviously, $\widehat Q(\eta)=0$, and Lemma \ref{lem_14}.2, implies that $\frac{r}{p^{\nu}}$ divides either $p^2-1$ or $p^2+1$. 

Very similar arguments apply to the other cases by taking into account that  {if $p|K$, then $\widehat Q(X)=\widehat L(X^{\frac{r+1}{p^{\nu}}}-1)^{2p^{\nu}}$ where $\nu =\nu_p(r+1)$, and if $p|L-K$, then $\widehat Q(X)=\widehat L (X-1) (X^{\frac{2r+1}{p^{\nu}}}-1)^{p^{\nu}}$ where $\nu =\nu_p(2r+1)$.}
\end{proof}
As an example we give an application of the previous Lemma in the case $p=2$.
\begin{corollary}\label{cor_6}
Let $\Gamma$ be an unbalanced distance-biregular graph with diameter $d=2r+2$, girth deficiency two  and vertex degrees $\beta_0^A=mK+1$ and $\beta_0^B=mL+1$, with $\GCD(K,L)=1$.
	If $L$ and $K$ are both odd, then the diameter of $\Gamma$ is either $4$ or $6$.
\end{corollary}
\begin{proof}
	If $L$ and $K$ are both odd then $L-K$ is even, that is $2| L-K$. By Lemma \ref{lem_12}, $2r+1\in U_2$. Since $2r+1$ is odd for any $r\in\bZ$, that is $\nu_2(2r+1)=0$, we have that $2r+1$ divides either $2^2-1=3$ or $2^2+1=5$.  {This yields} $r\leq 2$ and the result follows.
\end{proof}
\begin{remark}
{\em Note that, when Lemma \ref{lem_12} is applied to a single prime divisor $p$ of $KL(L-K)$ it does not  give any bound on the diameter since $U_p$ is an infinite set; in fact,  if $n\in U_p$ then  $n'=np^a\in U_p$, for all  positive integers $a$. Focusing instead on pairs of primes yields significant bounds.}
\end{remark}

\begin{lemma}\label{lem_17}
If $p$ and $q$ are distinct primes, then 	 $U_p\cap U_q$ is finite.
\end{lemma}
\begin{proof}
	Let $n\in U_p\cap U_q$ and write 
	\[
	n=p^aq^bs,
	\] where $p,q\nmid s$. Since $n\in U_p$, we have that $q^bs$ divides either $p^2-1$ or $p^2+1$. Therefore, $b$ and $s$ can take only a finite number of values. Similarly, from $n\in U_q$, we obtain that $a$ can take only a finite number of values.
\end{proof}
\begin{remark}\label{rem_13}
{\em The proof of Lemma \ref{lem_17} gives a powerful way to compute the greatest integer in  $U_p\cap U_q$ for any distinct primes $p$ and $q$. Take, for example,  $p=3$ and $q=7$. Let $n\in U_3\cap U_7$ and write it as $n=3^a7^bs$, with $3,7\nmid s$. Then, $7^bs$ must divide either $3^2-1=8$ or $3^2+1=10$. Similarly, $3^as$ must divide either $7^2-1=48$ or $7^2+1=50$. We have four cases to consider. If $7^bs\mid 8=2^3$ and $3^as\mid 48=2^43$ then $b=0$, $a\leq 1$ and $s|2^3$, so $n\le3\cdot2^3=24$. Similarly, if $7^bs\mid 8=2^3$ and $3^as\mid 50=2\cdot 5^2$ then $n\leq 2$; if $7^bs\mid 10=2\cdot 5$ and $3^as\mid 48=2^43$ then $n\leq 2\cdot 3=6$; finally, if $7^bs\mid 10=2\cdot 5$ and $3^as\mid 50=2\cdot 5^2$ then $n\leq 2\cdot 5=10$. Therefore, the greatest integer in $U_3\cap U_7$ is 24.}
\end{remark}

While Lemma \ref{lem_12} gives constraints on $r$ from prime numbers dividing
$KL(L-K)$, we now consider prime numbers not dividing $KL(L-K)$. Notice that
$2\mid KL(L-K)$: if either $K$ or $L$ is even, this is immediate, while if both
are odd, then $L-K$ is even. Thus, every prime $p$ not dividing $KL(L-K)$ is
odd.
The main idea is the following. For any prime
$p\nmid KL(L-K)$, the reduction $\widehat Q(X)$ modulo $p$ of the polynomial $Q(X)$ has degree $2r+2$ and its splitting field is contained in $\bF_{p^4}$ by Lemma \ref{lem_14}. This constrains the number of its zeros and hence the
value of $r$. The argument must take account of zeros with multiplicity
greater than one. The next lemmas bound the number of nonsimple zeros
of $\widehat Q(X)$ when $p\nmid KL(L-K)$. More precisely, we show that
$\widehat Q(X)$ has no zero with multiplicity greater than four and
construct polynomials $h_1(X),h_2(X),h_3(X),h_4(X)$ such that every
zero of multiplicity $m\le4$ is a zero of $\widehat h_m(X)$.
\begin{lemma}\label{lem_13'}
Let $K$ and $L$ be two integers  and set
\[
h_2(X)  =  (X-1) g_1(X) g_2(X),
\]
where	\[
	\begin{array}{rcl}
		g_1(X) & = &(r+1) LX- {rK},\\[.1in]
		g_2(X) & = & (2r+1)^2KLX^2-(4r^2(L^2+K^2)+8rL^2-2KL+4L^2)X+(2r+1)^2KL.
	\end{array}
	\]
	Let $p$ be  any odd prime number not dividing $KL(L-K)$. Then, each zero of $\widehat Q(X)\in\bZ_p[X]$ with multiplicity at least two is a zero of $\widehat h_2(X)\in\bZ_p[X]$.
\end{lemma}
\begin{proof}
	 {We recall that any zero of $\widehat Q(X)$ with multiplicity at least two is also a zero of its derivative $\widehat Q'(X)$.  The first derivative of $Q(X)$ is 
	\[
	Q'(X)= 2(r+1)LX^{2r+1}-(2r+1)KX^{2r}-2(r+1)(L-K)X^{r}-K.
	\]}
	To find the repeated zeros of $\widehat Q(X)$, we must simultaneously solve $\widehat Q(X)=0$ and $\widehat Q'(X)=0$, that is we are looking for the solutions of 
	\begin{equation*}
		\left\{\begin{array}{lcl}
			\widehat LX^{2r+2}-\widehat KX^{2r+1}-2 (\widehat L-\widehat K)X^{r+1}-\widehat KX+ \widehat L & = & 0 \\[.1in]
			2(\widehat r+1)\widehat LX^{2r+1}-(2\widehat{r}+1)\widehat K X^{2r}-2(\widehat r+1)(\widehat L-\widehat K)X^{r}-\widehat K & = & 0 .
		\end{array}\right.  
	\end{equation*}
	To simplify computations, we set $Y=X^r$ and $Z=X^{2r}$ in order to reinterpret the previous system of polynomial equations as a system of linear equations in the variables $Y$ and $Z$:
	\begin{equation}\label{eq_23}
		\left\{\begin{array}{rrcl}
			(\widehat LX^{2}-\widehat KX)Z & -2 (\widehat L-\widehat K)XY & = & \widehat KX- \widehat L \\[.1in]
			\left(2 (\widehat r+1)\widehat LX- (2\widehat{r}+1)\widehat K\right)Z & -2(\widehat r+1)(\widehat L-\widehat K)Y & = & \widehat K.
		\end{array}\right.  
	\end{equation}
	 {Letting $g_1(X)  = (r+1) LX- {rK}$, the  determinant of the above system is 
	\begin{gather*}
			\Delta = \begin{vmatrix}
				\widehat LX^{2}-\widehat KX & -2 (\widehat L-\widehat K)X \\
				2 (\widehat r+1)\widehat LX- (2\widehat{r}+1)\widehat K & -2(\widehat r+1)(\widehat L-\widehat K)
			\end{vmatrix}=2(\widehat L-\widehat K)\widehat g_1(X)X.	\end{gather*}}
	Since $p\nmid (L-K)$ and $\widehat Q(0)\neq 0$, we have that $\Delta\neq 0$, except for the  possible zeros of $\widehat g_1(X)$. Since $p$ does not divide either $K$ or $L$, and  $r+1$ and $r$ are coprime, we see that  $\widehat g_1(X)$ is not the zero polynomial in $\bZ_p[X]$.
	To shorten the notation, we write the solutions $(y,z)$ of \eqref{eq_23}, obtained by applying the well-known Cramer's rule, as  $(\Delta_Y/\Delta,\Delta_Z/\Delta)$,  {where 
	\begin{gather*}
			\Delta_Y= \begin{vmatrix}
				\widehat LX^{2}-\widehat KX & \widehat KX- \widehat L \\
				2 (\widehat r+1)\widehat LX- (2\widehat{r}+1)\widehat K & \widehat K
			\end{vmatrix}
			\end{gather*}
and   
	\begin{gather*}
			\Delta_Z= \begin{vmatrix}
				\widehat KX- \widehat L & -2 (\widehat L-\widehat K)X \\
				\widehat K & -2(\widehat r+1)(\widehat L-\widehat K)
			\end{vmatrix}.
	\end{gather*}}
	Since $Y^2=Z$, the polynomial 
	\[
	 \Delta^2(Y^2-Z) = \Delta_Y^2-\Delta_Z\Delta = \widehat{KL}(X-1)^2\widehat g_2(X),
	\]
	must be zero on the repeated zeros of $\widehat Q(X)$. This yields that all zeros of $\widehat Q(X)$ are simple, with exception of $X=1$ and, possibly, the zeros of $\widehat g_1(X)$ and $\widehat g_2(X)$. Therefore, the zeros with multiplicity at least two are also zeros of  $\widehat h_2(X)=(X-1) \widehat g_1(X) \widehat g_2(X)$.
\end{proof}

\begin{lemma}\label{lem_13''}
 {Let $K$ and $L$ be two integers  and set 
	\[
	\begin{array}{rcl}
		h_3(X) & = & g_1(X)\left[(2r+1)^2KLX^2+2(r(K+L)+L)^2X+(2r+1)^2KL\right],\\[.1in]
		h_4(X) & = & (r+1)^2LX-r^2K,
	\end{array}
	\]
	where $g_1(X)$ is as in Lemma \ref{lem_13'}.
Let $p$ be any odd prime not dividing $KL(L-K)$.}	If $p\mid(r+1)(2r+1)$, then $\widehat Q(X)\in\bZ_p[X]$ has no zero with multiplicity greater than two.  {If $p\nmid(r+1)(2r+1)$,} then every zero of $\widehat Q(X)$ with multiplicity at least three is a zero of $\widehat{h}_3(X)$, every zero of multiplicity four is a zero of $\widehat{h}_4(X)$ and there are no zeros with multiplicity greater than four.
\end{lemma}
\begin{proof}
	A zero of $\widehat Q(X)$ with multiplicity greater than two  is a zero of $\widehat Q'(X)\in\bZ_p[X]$ with  multiplicity greater than one, so it is a zero of $\widehat Q''(X)\in\bZ_p[X]$. We set
	\[
	Q^{[2]}(X):=Q''(X)/2X^{r-1}= (r+1)(2r+1)LX^{r+1}-r(2r+1)KX^{r}-r(r+1)(L-K).
	\]
	Since $X=0$ is not a root of $\widehat Q(X)$, the zeros of
$\widehat Q(X)$ with multiplicity greater than two are also zeros of
$\widehat Q^{[2]}(X)$.
	\\
Assume $p|(2r+1)$. Then, $\widehat Q^{[2]}(X)=\widehat{r}^2(\widehat{L}-\widehat{K})$. In addition,  $p\nmid r$ otherwise $p$ would divide $(2r+1) -2r=1$. This, together with $p\nmid\, L-K$, implies that $\widehat Q^{[2]}(X)$ is a non-zero constant polynomial, whence $\widehat Q(X)$ has no zeros with multiplicity greater than two.

 If $p|(r+1)$, then $\widehat Q^{[2]}(X)=-\widehat{K}X^r$. As $p\nmid K$ and $X=0$ is not a zero of $\widehat Q(X)$, we get that   $\widehat Q(X)$ has no zeros with multiplicity at least three.

Assume now  $p\nmid(r+1)(2r+1)$.
To find the zeros of $\widehat Q(X)$ with multiplicity at least three, we solve 
\[
\left\{\begin{array}{rcl}
\widehat Q(X) & = & 0\\[.03in]
\widehat Q'(X) & = & 0\\[0.03in]
\widehat Q^{[2]}(X) & = & 0
\end{array}\right.
\]
 and proceed as in the proof of Lemma \ref{lem_13'} setting $Y=X^r$ and $Z=X^{2r}$. From the first two equations we obtain $Y=\Delta_Y/\Delta$, while the last equation becomes
	\[
	(\widehat{r}+1)(2\widehat{r}+1)\widehat LXY-\widehat{r}(2\widehat{r}+1)\widehat KY-\widehat{r}(\widehat{r}+1)(\widehat{L}-\widehat{K})=0.
	\]
	Then, by multiplying this equation by $\Delta$ and taking into account that $\Delta Y=\Delta_Y$, we obtain
\[
\begin{array}{rcl}
0 & = & \left((\widehat{r}+1)\widehat{L}X-\widehat{r}\widehat{K}\right)\left((2\widehat{r}+1)^2\widehat{K}\widehat{L}X^2+2(\widehat{r}(\widehat{K}+\widehat{L})+\widehat{L})^2X+(2\widehat{r}+1)^2\widehat{K}\widehat{L}\right)\\[0.03in]
& = & \widehat h_3(X).
\end{array}
\]

Therefore, the zeros of $\widehat Q(X)$ with multiplicity at least three are also  zeros of $\widehat h_3(X)$.

Similarly, the zeros of $\widehat Q(X)$ with multiplicity at least four  are multiple zeros of $\widehat Q^{[2]}(X)$. So, considering the polynomial
\[
 Q^{[3]}(X)=Q^{[2]'}(X)/X^{r-1}=(2r+1)\left((r+1)^2LX-r^2K\right)=(2r+1)h_4(X),
\]
we have that every such zero is also a zero of $\widehat Q^{[3]}(X)$, hence of $\widehat h_4(X)$ since $p\nmid (2r+1)$. 

Finally, every zero of $\widehat Q(X)$ with multiplicity greater than four would be a non-simple zero of $\widehat h_4(X)$. Therefore, it would be a zero of its derivative $\widehat h'_4(X)=(\widehat r+1)^2L$. Since $p\nmid (r+1)$, $\widehat h'_4(X)$ is a non-zero constant polynomial. Hence, $\widehat Q(X)$ has no zero with multiplicity greater than four.
\end{proof}
The next lemma is a consequence of the two previous lemmas.
\begin{lemma}\label{lem_13}
	Let $\Gamma$ be an unbalanced distance-biregular graph with diameter $d=2r+2$, girth deficiency two  and vertex degrees $\beta_0^A=mK+1$ and $\beta_0^B=mL+1$, with $\GCD(K,L)=1$. Then, for any odd prime $p$ not dividing $KL(L-K)$, the polynomial $\widehat Q(X)\in\bZ_p[X]$ divides $\widehat h_1(X)\widehat h_2(X)\widehat h_3(X)\widehat h_4(X)$, where 
	\[
	h_1(X)  = \frac{(X^{p^2+1}-1)(X^{p^2-1}-1)}{(X^2-1)},
	\]
	and $h_2(X)$, $h_3(X)$ and $h_4(X)$ are the polynomials defined in Lemmas \ref{lem_13'} and \ref{lem_13''}.
\end{lemma}
\begin{proof}
	By Lemma \ref{lem_14}, the zeros of $\widehat Q(X)$ are zeros of $(X^{p^2-1}-1)(X^{p^2+1}-1)$, hence they are zeros of $\widehat h_1(X)$. Note that the division by $X^2-1$ has been done to obtain that every zero of $\widehat h_1(X)$ is simple.  {Therefore, combining Lemmas \ref{lem_13'} and \ref{lem_13''}, we find that  every zero of $\widehat{Q}(X)\in\bZ_p[X]$ has multiplicity at most four, and a zero of multiplicity $m\le 4$  is a zero of $\widehat{h}_m(X)$. Thus every irreducible factor of $\widehat Q(X)$ in the splitting field divides $\widehat h_1(X)\widehat h_2(X)\widehat h_3(X)\widehat h_4(X)$.}
\end{proof}
\begin{proposition}\label{prop_9}
Let $\Gamma$ be an unbalanced distance-biregular graph with diameter $d=2r+2$, girth deficiency two  and vertex degrees $\beta_0^A=mK+1$ and $\beta_0^B=mL+1$, with $\GCD(K,L)=1$. Then,  $r\leq p^2+2$  for any prime $p$ not dividing $KL(L-K)$.
\end{proposition}
\begin{proof}
From Lemmas \ref{lem_13'},  \ref{lem_13''} and \ref{lem_13}, we see that 
$\deg\, h_1(X)  = 2p^2-2$, $\deg\, h_2(X)  =  4$, $\deg\, h_3(X)  =  3$ and $\deg\, h_4(X)  =  1$.
Therefore, from Lemma \ref{lem_13} we get \[2r+2=\deg\, q(X)\le \deg\, h_1(X)h_2(X)h_3(X)h_4(X)\leq 2p^2+6.\]
\end{proof}
Finally, we can collect all the previous results to give our main result on the diameter of an unbalanced distance-biregular graph with girth deficiency two. By Lemma \ref{lem_1}, the diameter of such graph is even and we write it as $d=2r+2$. As usual we write the vertex degrees as $\beta_0^A=mK+1$ and $\beta_0^B=mL+1$, with $\GCD(K,L)=1$.

Once we fix a prime $p$, Lemmas \ref{lem_12} and \ref{lem_13} give some conditions that $r,K,L$ must satisfy. It turns out that to obtain a bound on the diameter it is enough to consider the first six primes; any other larger prime does not improve the bound in Proposition \ref{prop_7''}. 
In the next lemma we take advantage of the fact that, for any given
prime, the condition of Lemma \ref{lem_13} can be checked by a finite
computation.
\begin{lemma}\label{prop_7'}
	Let $\Gamma$ be an unbalanced distance-biregular graph with diameter $d= 2r+2$, girth deficiency two and vertex degrees $\beta_0^A=mK+1$ and $\beta_0^B=mL+1$, with $\GCD(K,L)=1$. If $r\geq7$, then $KL(L-K)$ is divisible by the primes $2,3,5,7,11,13$.
\end{lemma}
\begin{proof}
	Since $2$ always divides $KL(L-K)$, it remains to consider odd primes. Let $p$ be an odd prime not dividing $KL(L-K)$. By Proposition \ref{prop_9}, there is a finite number of possibilities for the polynomial $\widehat Q(X)$. Indeed, the proposition gives  $r\leq p^2+2$, and $p \nmid KL(L-K)$ gives  $\widehat K\neq 0\neq \widehat L$ and $\widehat K\neq \widehat L$. Since $\widehat K,\widehat L$ belong to the finite field $\bZ_p$, there are at most $(p-1)(p-2)(p^2+1)$ possibilities for $\widehat Q(X)$. For each $p\in\{3,5,7,11,13\}$, we handed all the possibilities to the computer algebra system MAGMA \cite{magma} and we obtained that in any case $\widehat{Q}(X)$ does not divide $\widehat h_1(X)\widehat h_2(X)\widehat h_3(X)\widehat h_4(X)$ whenever $r\geq7$. By Lemma \ref{lem_13}, if $r\geq7$ each prime $p\leq 13$ divides $KL(L-K)$.
\end{proof}

The previous lemma allows us to obtain a first bound on the diameter by using the divisibility hypothesis in Lemma \ref{lem_12}.

\begin{proposition}\label{prop_7''}
	Let $\Gamma$ be an unbalanced distance-biregular graph with diameter $d= 2r+2$ and girth deficiency two. Then $r\leq 50$.
\end{proposition}
\begin{proof}
	Let the vertex degrees of the graph be $\beta_0^A=mK+1$ and $\beta_0^B=mL+1$, with $\GCD(K,L)=1$ and assume $r\geq 7$. By Lemma \ref{prop_7'}, each $p\in\{2,3,5,7\}$ divides $KL(L-K)$. By the pigeonhole principle, at least two such primes satisfy the same divisibility hypothesis in Lemma \ref{lem_12}. This means that there exist $n\in\{r,r+1,2r+1\}$ and $p,q\in \{2,3,5,7\}$ such that $n\in U_p\cap U_q$.  Each of the sets $U_p\cap U_q$ with $p,q\in \{2,3,5,7\}$ is finite by Lemma \ref{lem_17}. A direct computation, as in Remark \ref{rem_13}, gives 
	\[ \begin{array}{ccc}
		\max (U_2\cap U_3) = 24,& \max (U_2\cap U_5) = 40,& \max (U_2\cap U_7) = 48, \\
		\max (U_3\cap U_5) = 30,& \max (U_3\cap U_7) = 24,& \max (U_5\cap U_7) = 50.
	\end{array}    \]
	Therefore, in any case, we obtain $r\leq n\leq  50$.
\end{proof}

In our main theorem we further restrict the possibilities for the diameter of $\Gamma$.

\begin{theorem}\label{prop_7}
 There do not exist unbalanced distance-biregular graphs with diameter $2r+2$ and girth $4r+2$ for $13\le r\le 22$ and $r\ge 25$.
\end{theorem}
\begin{proof}
By Proposition \ref{prop_7''}, it is enough to check for $ r\leq 50$. By Lemma \ref{prop_7'} and Lemma \ref{lem_12}, for any $p\in\{2,3,7,13\}$ and $r\geq 7$ we have $\{r,r+1,2r+1\}\cap U_p\neq \varnothing$. By a direct computation we find that this condition is satisfied only if $r\leq 12$, $r=23$ or $r=24$. We remark that checking this condition for other values of $p$ won't restrict further the possible values of $r$. This is because 24 always divides $p^2-1$ and 10 divides either $p^2-1$ or $p^2+1$. Therefore, for any remaining $r$ and any prime $p\geq 7$, the set $\{r,r+1\}\cap U_p$ is always non-empty.
\end{proof}

\section{Feasible parameters}\label{sec_7}

The aim of this section is to determine the triples $(r,k,l)$ for which there could exist an unbalanced distance-biregular graph with diameter $d=2r+2$, girth deficiency two and degrees $k+1$ and $l+1$.  We will use some of the results from the previous sections, which we collect in the following theorem.		

\begin{theorem}\label{th_7}
		Let $\Gamma=(V,E)$ be an unbalanced distance-biregular graph with diameter $d=2r+2$, girth deficiency two and vertex degrees $k+1$ and $l+1$. Then, the parameters $r$, $k$ and $l$ satisfy the following feasibility conditions:
		\begin{enumerate}
			\item $1\le r\le 12$ or $r\in\{23, 24\}$; \label{item_1}
			\item $k<l$; \label{item_2}
			\item $l^r(l+1)\equiv 0\,\mod\, (k+1)$; \label{item_3}
			\item the irreducible factors of $F_r(X)$ have degree at most two; \label{item_4}
			\item for each zero $t$ of $F_{r}(X)$, $m_t=|E|\frac{(t^2-4kl)}{(t-kl-1)((2r+1)(t+k+l)+l-k)}$ is a positive integer. \label{item_5}
		\end{enumerate}
\end{theorem}
\begin{proof}
	 Condition \ref{item_1} is in Theorem \ref{prop_7}, Conditions \ref{item_2} and \ref{item_3} are in Corollary \ref{cor_2}, Condition \ref{item_4} is in Corollary \ref{cor_4} and Condition \ref{item_5} is in Theorem \ref{th_5}.
\end{proof}
	
A triple $(r,k,l)$ that satisfies the conditions of Theorem \ref{th_7} is said to be {\em feasible}.

We remark that Conditions \ref{item_1} and \ref{item_4} are redundant since they follow from the others. We included them anyway since in practical applications they are easier to check than Condition \ref{item_5}.

Note that Condition \ref{item_5} is actually a collection of $r$ conditions, one for each zero of $F_{r}(X)$. Intuitively, this means that  the larger the diameter of an unbalanced distance-biregular graph with girth deficiency two, the more restrictions we have on its parameters $k$ and $l$. When the diameter is sufficiently high, there are so many restrictions that such a graph cannot exist.  This is the intuitive explanation behind Condition \ref{item_1}. It also suggests that, for any admissible value of $r$ it is possible to obtain more precise conditions on $k$ and $l$. This is the goal of the rest of this section.

\subsection{Unbalanced distance-biregular graphs  with diameter 4 and girth 6}

{The case of $r=1$ corresponds to unbalanced distance-biregular graphs of diameter four and girth six, according to the previous notation. In this case, the algebra $\cA_4(k,l)$ has, up to isomorphism, a unique two-dimensional representation corresponding to the  root  $t=-k$ of $F_1(X)=X+k$. Its multiplicity, according to Theorem \ref{th_5}, is $m_t=k(l+1)$. Therefore, the only meaningful conditions of Theorem \ref{th_7} are Conditions \ref{item_2} and \ref{item_3}. In any case, unbalanced distance-biregular graphs with diameter four and girth six arise precisely as incidence graphs of Steiner systems of type $S(2,k,v)$. There is extensive literature on them; see \cite{han,kir} and \cite[Appendix A]{lato2}, to name a few.}

\subsection{Unbalanced  distance-biregular graphs  with diameter 6 and girth 10}

The case of $r=2$ corresponds to unbalanced distance-biregular graphs of diameter six and girth ten. 

\begin{proposition}\label{prop_8}
	Let $\Gamma$ be an unbalanced distance-biregular graph with diameter six, girth ten and vertex degrees $k+1$ and $l+1$. Then $k^2+4kl$ is a perfect square. 
\end{proposition}
\begin{proof}
	In this case we have $r=2$. By Theorem \ref{th_5}, taking into account Proposition \ref{cor_1}, Part 3, for each zero $t$  of $F_2(X)$, the multiplicity $m_t$ associated with $t$ has the form
	\begin{equation}\label{eq_32}
		m_t=\frac{(k^3l^3+k^3l^2-k-1)(l+1)}{kl-1}\frac{(t^2-4kl)}{(t-kl-1)(5(t+k+l)+l-k)}.
	\end{equation}
	
	The roots of $F_2(X)=X^2+kX-kl$ are $t_{\pm}=\frac{-k\pm\sqrt{k^2+4kl}}{2}$. Assume that $k^2+4kl$ is not a perfect square, that is $t_{\pm}$ are both irrational numbers. By plugging  {the explicit values of $t_{\pm}$} in \eqref{eq_32} and handing it to the computer algebra system MAGMA \cite{magma}, we get
	\[
	m_{t_{\pm}}=k\frac{(-kl^2-kl+2l^2+l-1)t_{\pm}+2kl^3+3kl^2+kl+2l^2+2l}{k+4l}.
	\]
	It is evident that $m_{t_{\pm}}$ is a rational number if and only if $k(-kl^2-kl+2l^2+l-1)=0$. This yields, either $k=0$  or $l(-kl-k+2l+1)=1$. Since we assume that $l>k\ge 1$, none of these cases occurs. 
\end{proof}

Using MAGMA \cite{magma}, we first searched for pairs $(k,l)$, with $1\le k<l<10^5$, that satisfy the feasibility condition given by Proposition \ref{prop_8}, then Conditions 2 and 3 provided by Theorem \ref{th_7}. For each pair that passes the previous test and for  all zeros of $F_2(X)$, we verified if Condition 5 of Theorem \ref{th_7} is also satisfied.
Only 554 pairs passed the complete test. In particular, for
$1\le k<l<100$ we found the 17 pairs
$(1,2)$, $(1,6)$, $(1,56)$, $(3,6)$, $(4,15)$, $(4,24)$, $(4,35)$,
$(4,80)$, $(4,99)$, $(5,30)$, $(7,84)$, $(8,30)$, $(9,54)$,
$(11,66)$, $(14,84)$, $(24,90)$ and $(27,84)$.

The pairs $(1,2)$, $(1,6)$ and $(1,56)$ correspond to the case in
which one bipartition class has valency two; distance-biregular graphs
with 2-valent vertices were classified by Mohar and Shawe-Taylor in
\cite{mst}. Delorme \cite{del} gives an example of a distance-biregular graph with diameter $6$ and girth $10$ and $(k,l)=(1,2)$, namely the subdivision graph of the Petersen graph. The pair $(1,6)$ corresponds to the subdivision graph of the Hoffman-Singleton graph, and the pair $(1,56)$ would be given by the subdivision graph of the missing Moore graph if such a graph exists.

More generally, the computational results above raise the question of which of the feasible triples $(2,k,l)$ can actually be realized by unbalanced distance-biregular graphs. This leads to the following problem:
	
\begin{problem*}\label{conj_2}
For each of the above feasible triples $(2,k,l)$, determine whether there exists an unbalanced distance-biregular graph with diameter $6$, girth $10$ and degrees $k+1$ and $l+1$, and construct one whenever it exists.
\end{problem*}

\subsection{Unbalanced distance-biregular graphs with girth deficiency two and diameter at least 8}

The case of $r\geq 3$ corresponds to unbalanced distance-biregular graphs with girth deficiency two and diameter at least 8. In this case, we used the computer to search for feasible triples $(r,k,l)$.  A naive approach would be to examine all possible pairs $(k,l)$ with $1\le k<l$ and $l$ below a fixed upper limit, and then verify the conditions of Theorem \ref{th_7}. 
This computation requires many polynomial factorizations, so it is too time-consuming and would only be applicable for a relatively small upper limit. We therefore adopted a different approach to reduce the calculation time.

The main idea behind our algorithm is the following: for each prime  $p$, Lemma \ref{lem_14} provides certain conditions on $r$, $k$, and $l$. This gives us a first list of feasible triples $(r,k,l)$ to which  Theorem \ref{th_7} will be applied. Using the Chinese Remainder Theorem, we can combine the conditions relating to different prime numbers to produce a shorter list.

Our algorithm takes as input three positive integers $r$, $\BOUND$ and  $P_r$. The output consists of all pairs $(k,l)$ such that $(r,k,l)$ is
feasible and $l/\GCD(k,l)<\BOUND$.  The parameter $P_r$ represents the upper bound for the prime numbers for which we consider the condition given by Lemma \ref{lem_14}.
Moreover, for each $a\in \bZ$, we consider the polynomial 
\[
Q_a(X):=(a+1)(X^{2r+2}+1)-aX(X^{2r}+1)- 2X^{r+1}.
\]

The algorithm consists of the following steps: \begin{enumerate}
	\item \label{step_1} Take as input three positive integers: $r\ge 3$, $\BOUND$, $P_r$.
	\item \label{step_2} Select all primes $p\le P_r$ such that $2r+1\not\in\bigcup_{i=1}^{s}{U_{p_i}}$, and denote
them by $p_1,\ldots,p_s$.
	\item \label{step_3}  For each $p_i$, $i=1,\ldots,s$, initialize
$T_r^{(i)}=\varnothing$. For every $a\in\bZ_{p_i}$, consider the
polynomial $\widehat Q_a(X)=(a+1)(X^{2r+2}+1)-aX(X^{2r}+1)-2X^{r+1}
\in\bZ_{p_i}[X]$. Check whether every non-zero root of $\widehat Q_a(X)$ is also a zero of $(X^{p_i^2-1}-1)(X^{p_i^2+1}-1)$. 
If this condition is satisfied, add $a$ to $T_r^{(i)}$. Finally, write $T_r^{(i)} =\{a_1^{(i)},\ldots,a_{j_{p_i}}^{(i)}\}$.
	\item \label{step_4} Set $n=p_1\cdots p_s$. For each $s$-tuple $\sigma=(a_1,\ldots, a_s)\in T_r^{(1)}\times\cdots\times T_r^{(s)}$, use the Chinese Remainder Theorem to find the unique integer $\alpha_\sigma$ with $0\le\alpha_\sigma<n$  such that $\alpha_\sigma\,\equiv a_i\,\mod\, p_i$, for $i=1,\ldots,s$. Collect these integers in the set  $\cT_r$.
\item \label{step_5}
	For each $\alpha\in \mathcal{T}_r$ and  $\delta\in\mathbb{Z}$ such that $1\leq \delta < \BOUND$ do the following: 
	\begin{enumerate}
		\item Compute \begin{equation}\label{eq_42}
			\begin{split}
				\beta_{\rm MIN} &= \frac{1 - \alpha\delta}{n} \\
				\beta_{\rm MAX} &= \frac{\BOUND - (\alpha+1) \delta }{n}.
			\end{split}
		\end{equation}
	\item For each positive integer $\beta$ with $\beta_{\rm MIN} \le \beta < \beta_{\rm MAX}$, compute $(K,L)$ using the formulas
			\begin{equation}\label{eq_35}
				K=\alpha\delta+\beta n, \qquad  L=(\alpha+1)\delta+\beta n.
			\end{equation}
		 If $\GCD(K,L)=1$ collect the pair $(K,L)$ into the set $\cS_r$.
		 \end{enumerate} 
	\item \label{step_6} For each $(K,L)\in \cS_r$ compute the irreducible  factorization of $Q(X)/(X-1)^2$ over the integers, where  $Q(X)$ is the polynomial as in \eqref{eq_44}:
	\[
Q(X)=LX^{2r+2}-KX^{2r+1}-2(L-K)X^{r+1}-KX+ L.	
	\] 
If at least one of the factors has degree other than 2 or 4, remove $(K,L)$ from  $\cS_r$.

	\item \label{step_7} For each $(K,L)\in \cS_r$ and each divisor $D>1$ of $L^r(L-K)$, check whether $K$ divides $D-1$. If so, set $m=\frac{D-1}{K}$, $(k,l)=(mK,mL)$ and collect $(k,l)$ in the set $\bar\cS_r$.
	\item \label{step_8} For each $(k,l)\in\bar\cS_r$ compute the polynomial $F_r(X)$ using Eq. \eqref{eq_7}. For each zero $t$ of $F_r(X)$, compute the multiplicity $m_t$ using  Theorem \ref{th_7}, Part 5. If at least one of this multiplicities is not an integer,  remove  $(k,l)$ from $\bar\cS_r$.
	\item \label{step_9} Return $\bar\cS_r$ as output.
\end{enumerate}

We now prove that the above algorithm produces the desired output. Fix three positive integers $r\ge3$, $\BOUND$ and $P_r$ and assume that $(r,k,l)$ is  a feasible triple with $l/\GCD(k,l)<\BOUND$. We claim that $(k,l)$ is in the set $\bar\cS_r$ given as output.

Set $(k,l)=(mK,mL)$, where $m=\GCD(k,l)$. Note that $\GCD(K,L)=1$ and, by Condition \ref{item_2} of Theorem \ref{th_7}, $K<L$. 

Let $p_1,p_2\ldots, p_s$ be the prime numbers found in Step \ref{step_2}. Since $2r+1\not\in U_{p_i}$, Lemma \ref{lem_12}, Part 3, implies that
$p_i\nmid(L-K)$ for $i=1,\ldots,s$. Set  $n=p_1\cdots p_s$ as in Step \ref{step_4}. Since $\GCD(n, L-K)=1$, then there is an integer $\alpha$ such that $\alpha(L-K) \equiv K\ \mod\ n$. For each $i=1,\dots,s$, making computations modulo $p_i$, we have $\widehat \alpha=\widehat K/(\widehat L-\widehat K)\in \bZ_{p_i}$, $\widehat \alpha+1= \widehat L/(\widehat L-\widehat K)\in \bZ_{p_i}$ and $ \widehat Q(X)/(\widehat L-\widehat K) = \widehat Q_\alpha(X)\in \bZ_{p_i}[X]$. By Lemma \ref{lem_14}, every non-zero root of $\widehat Q_\alpha(X)\in\bZ_{p_i}[X]$ is also a zero of $(X^{p_i^2-1}-1)(X^{p_i^2+1}-1)$. Therefore, by Step \ref{step_3}, $\alpha \equiv a_i$ for some $a_i\in T_r^{(i)}$. Since this holds for every $i=1,\dots,s$, by Step \ref{step_4} we have that $\alpha\in \cT_r$.

By the construction of $\alpha$, we can write $K=\alpha(L-K)+\beta n$ for some $\beta\in \bZ$. Then we see that  Eq. \eqref{eq_35} holds with $\delta=L-K$. Moreover, by taking into account $1 \le K< L< \BOUND$ we get $1\leq \delta < L < \BOUND$ and  $\beta_{\rm MIN} \le \beta < \beta_{\rm MAX}$. Therefore, since  $\GCD(K,L)=1$, we see that $(K,L)$ belongs to the set $\cS_r$ obtained in Step \ref{step_5}.

By Lemma \ref{prop_6},  each irreducible factor of $Q(X)/(X-1)^2$ has degree 2 or 4; hence, in Step \ref{step_6} the pair $(K,L)$ is not removed from $\cS_r$.

By the proof of Corollary \ref{cor_3}, $D=k+1$ divides $L^r(L-K)$, and, clearly, $K$ divides $mK=k=D-1$. Recall that $(k,l)=(mK,mL)$, so $m=\frac{D-1}{K}$. Thus, $(k,l)$ is collected into  the set $\bar \cS_r$ in Step \ref{step_7}.

Finally, as $(r,k,l)$ satisfies Conditions \ref{item_4} and \ref{item_5} of Theorem \ref{th_7}, the pair $(k,l)$ is not discarded in Step \ref{step_8}, so it is in the set $\bar\cS_r$ given as output.

\begin{remark}
	{\em We note that the parameter $P_r$ affects the speed of computation only, not the results obtained. Indeed, in the above algorithm, the set of primes $p_1,\ldots,p_s\le P_r$ is used to construct the set $\cT_r$ in Step \ref{step_4}, which in turn is used to construct the set $\cS_r$ in Step \ref{step_5}. It is desirable that the size of $\cS_r$ is as low as possible to improve the execution speed of the subsequent steps. Increasing the number of primes considered generally increases
$|\cT_r|$. This may slow the computation, because Step \ref{step_5}
loops through all elements of $\cT_r$ to construct $\cS_r$. On the
other hand, the number $n=p_1\cdots p_s$ appears in the denominator in
Eq. \eqref{eq_42}. Consequently, using more primes shortens the
interval $[\beta_{\rm MIN},\beta_{\rm MAX}]$ and helps reduce the
cardinality of $\cS_r$. For this reason the value of $P_r$ was selected on
a case-by-case basis to balance these two effects and to have a reasonable execution speed.}
\end{remark}

\begin{remark}{\em Note that, if $r\leq 2$, Step \ref{step_2} of the above algorithm cannot be performed. Moreover, if $r$ is not one of the feasible values given in Condition \ref{item_1} of Theorem \ref{th_7}, then the output is guaranteed to be the empty set.
}
\end{remark}

Table \ref{results} shows the values of $P_r$ and $\BOUND$  that we were able to test on our computer, for every feasible $r\ge 3$ as given in Theorem \ref{prop_7}. The fourth row of the table shows the number $s_r^{(5)}$ of elements in $\cS_r$ at the end of Step \ref{step_5}. This is the number of candidates for the pair $(K,L)$: such pairs exist only for $r=3,4,5,7$.  Note that, even in these four cases, $s_r^{(5)}$ is significantly less than the number of all possible pairs $(K,L)$ with $L\leq B_r$. The fifth row shows the number $s_r^{(6)}$ of elements in $\cS_r$ at the end of Step \ref{step_6}, i.e., the number of pairs $(K,L)$ such that the polynomial $Q(X)$ has the desired irreducible factorization. When $r= 5$ or $7$, we did not find any such pair. Similarly, the last two rows of the table show the numbers $s_r^{(7)}$ and $s_r^{(8)}$ of the elements in $\bar \cS_r$ at the end of Steps \ref{step_7} and \ref{step_8}, respectively.  Even for $r=3$ or $4$, we did not find any feasible triple $(r,k,l)$, since all pairs found in Step \ref{step_7} produce non-integer multiplicities in Step \ref{step_8}. 
		{\renewcommand{\arraystretch}{1.7}
\begin{table}[htbp]
\centering
\caption{Computational results}
\label{results}
\begin{tabular}{c|*{12}{c}}
$r$ & 3 & 4 & 5 & 6 & 7 & 8 & 9 & 10 & 11 & 12 & 23 & 24 \\
\hline
$P_r$ & 23 & 23 & 37 & 47 & 47 & 47 & 47 & 47 & 47 & 47 & 59 & 59 \\
\hline
$\BOUND$ & $10^4$ & $10^4$ & $10^4$ & $10^5$ & $10^6$ & $10^6$ &
$10^7$ & $10^7$ & $10^8$ & $10^8$ & $10^9$ & $10^9$ \\
\hline
$s_r^{(5)}$ & 1632706 & 164557 & 45 & 0 & 150 & 0 & 0 & 0 & 0 & 0 & 0 & 0 \\
\hline
$s_r^{(6)}$ & 220 & 5 & 0 & 0 & 0 & 0 & 0 & 0 & 0 & 0 & 0 & 0 \\
\hline
$s_r^{(7)}$ & 1951 & 27 & 0 & 0 & 0 & 0 & 0 & 0 & 0 & 0 & 0 & 0 \\
\hline
$s_r^{(8)}$ & 0 & 0 & 0 & 0 & 0 & 0 & 0 & 0 & 0 & 0 & 0 & 0
\end{tabular}
\end{table}
}	

The above computational results suggest the following conjecture:
\begin{conjecture*}
There is no unbalanced distance-biregular graph with diameter greater than or equal to eight and girth deficiency two. 
\end{conjecture*}
\begin{remark}
	{\em Our programs, written in MAGMA \cite{magma}, were run on a Windows
system with an Intel Core i7 processor. The code is available from the
first author upon request by email.}
\end{remark}
 \section*{Acknowledgments}
\noindent The authors would like to thank the referees for their  time and
useful comments that improved the first version of the paper.

\end{document}